\documentclass[runningheads,envcountsame]{llncs}
\usepackage[T1]{fontenc}
\usepackage{graphicx}
\usepackage{hyperref}
\usepackage{thm-restate}
\usepackage{amsmath}
\usepackage{mathptmx}%
\usepackage{longtable}
\usepackage{booktabs}
\usepackage{fancyvrb}
\usepackage{xspace}
\usepackage{tikz}
\usepackage{makecell}
\usepackage{float}
\usepackage{listings}
\usepackage{upquote}
\usepackage[capitalize]{cleveref}
\usepackage{caption}
\usepackage{subcaption}
\usepackage{amsfonts}
\usepackage{amssymb}
\usepackage{xtab,multirow}
\usepackage{algpseudocodex}
\usepackage{multicol}
\usetikzlibrary{calc,shapes,decorations,decorations.pathreplacing}

\usepackage{todonotes}

\usepackage{color}

\newfloat{lstfloat}{htbp}{lop}
\floatname{lstfloat}{Listing}

\crefalias{lstfloat}{listing}

\newcommand\ie{i.e\@., }

\newcommand{\Wlog}{W.\,l.\,o.\,g.\xspace}
\newcommand{\sse}{\subseteq}
\newcommand{\cN}{\mathcal{N}}
\newcommand{\WP}{\mathrm{WP}}

\newcommand{\marky}{y}

\spnewtheorem{observation}{Observation}{\itshape}{}
\crefname{observation}{Observation}{Observation}
\Crefname{observation}{Observation}{Observation}

\makeatletter
\newcommand\footnoteref[1]{\protected@xdef\@thefnmark{\ref{#1}}\@footnotemark}
\makeatother

\begin{document}

\newcommand\relatedversion{}

\title{Geodetic Graphs: Experiments and New Constructions\relatedversion}

\author{Florian Stober \and Armin Wei\ss}
\authorrunning{F.\ Stober \and A.\ Wei\ss}
\institute{FMI, University of Stuttgart, Germany}

\maketitle

\begin{abstract}
  In 1962 Ore initiated the study of geodetic graphs.
  A graph is called geodetic if the shortest path between every pair of vertices is unique.
  In the subsequent years a wide range of papers appeared investigating their peculiar properties.
  Yet, a complete classification of geodetic graphs is out of reach.

  In this work we present an exhaustive search algorithm for enumerating all biconnected geodetic graphs of a given size.
  Using our program, we succeed to find all geodetic graphs with up to 25 vertices and all regular geodetic graphs with up to 32 vertices.
  This leads to the discovery of two new infinite families of geodetic graphs.
\end{abstract}

\section{Introduction}\label{sec:intro}

An undirected graph is called geodetic if for any pair of vertices the shortest path between them is unique.
Research on geodetic graphs began in 1962, when Ore posed the problem of classifying all such graphs~\cite{ore1962theory}.
This goal has been achieved for planar geodetic graphs and geodetic graphs of diameter two~\cite{stemple1974geodetic,stemple1968planar}.
Yet, after decades of active research, a full classification of finite geodetic graphs or any characterization by other means has not been attained.
Recently, there has been renewed interest in the problem, with several publications attacking the problem from different angles~\cite{EtgarN24,wrap108882,frasser2020geodetic} and an application of geodetic graphs in the construction of so-called $f$-edge fault-tolerant distance sensitivity oracles~\cite{BiloCCC0KS23}.

The present work complements this by studying small geodetic graphs in the hope to uncover more of their mysterious properties.
Since there are few known geodetic graphs, the study of small geodetic graphs seems important to us as it might lead to further insights about more general constructions.
Indeed, we discovered two new infinite families of geodetic graphs.
Moreover, we hope that having a complete list of small geodetic graphs readily available will be useful for and inspire future research.
Since a graph is geodetic if and only if each maximal biconnected subgraph is geodetic, the focus is on biconnected geodetic graphs.

\paragraph{Contribution.}
Our main contribution is a complete list of biconnected geodetic graphs with at most 25 vertices.
The scarcity of geodetic graphs is reflected in our results: the list contains only 149 graphs compared to the vast amount of far above $10^{50}$ biconnected graphs of up to 25 vertices.
Most of the geodetic graphs we obtained are subdivisions of smaller geodetic graphs, which were already known to be geodetic.
However, our list includes two new graphs, $F_5$ and $H(2,2,2,0)$, that, to the best of our knowledge, cannot be build by any of the previously known constructions.
We study their relations to existing constructions and succeed to find two new infinite families of geodetic graphs based on them.

To compute our list of geodetic graphs, we use a novel algorithmic approach to the problem of enumerating geodetic graphs.
The main idea behind our algorithm is to enumerate all shortest-path trees and then, by a backtracking search, fill in the structure of the graph with the remaining edges.
After each step, we perform pruning operations based on several theoretical results in order to check whether the graph can still be extended to a geodetic graph.
In addition, we use the graph isomorphism tool \texttt{nauty}~\cite{mckay2014practical} to identify isomorphic branches in the search space.

Finally, we consider the special case of regular graphs.
It has been conjectured, that a finite Cayley graph is geodetic if and only if it is either complete or an odd cycle~\cite[Conjecture 6]{wrap108882}.
This gives new interest in the special case of regular geodetic graphs.
Using our search algorithm, we are able to show that, up to 32 vertices, there are only the previously known geodetic graph: odd cycles, complete graphs, the Petersen graph and one graph with 28 vertices discovered in~\cite{bosak1978geodetic}.

\paragraph{Related Work.}
So far a characterization of geodetic graphs has been achieved for only two classes of graphs.
Stemple and Watkins, both students of Ore, showed that a planar graph is geodetic if and only if each maximal biconnected subgraph is either a single edge, a cycle of odd length, or a geodetic graph homeomorphic to the complete graph $K_4$~\cite{stemple1968planar} (meaning that it is a subdivision of $K_4$).
Stemple also obtained a characterization of geodetic graphs of diameter two, by describing vertex degrees and number of cliques~\cite{stemple1974geodetic}.
Later an equivalent characterization in terms of so-called $\pi$-spaces has been obtained~\cite{scapellato1986geodetic} and explicit constructions using projective and affine planes were discovered~\cite{blokhuis1988geodetic}.

A related question asks which geodetic graphs are homeomorphic to a given geodetic graph.
Starting with a complete graph $K_n$, labelling its vertices with integers and then subdividing each edge according to the sum of labels of its endpoints, we obtain a homeomorphic geodetic graph~\cite{plesnik1977two}.
In fact, this construction gives us all geodetic graphs homeomorphic to $K_n$~\cite{stemple1979geodetic}.
The Widespread Petersen graph $\WP_k$ is obtained from the Petersen graph:
Each edge connecting the inner to the outer pentagon is subdivided $k$ times.
The graph $\WP_k$ is geodetic~\cite{plesnik1977two}; however, there are also other geodetic graphs homeomorphic to the Petersen graph~\cite{frasser2020geodetic}.
So far we have not mentioned a very simple way of obtaining a geodetic graph from a given one:
Subdivide each edge with an even number of vertices~\cite{parthasarathy1982some}.
In fact, that result is a special case of a more general construction:
from a geodetic graph $G$ and a clique cover of the edge set the authors of~\cite{parthasarathy1982some} construct a geodetic graph $G^*$.
However, in the general case, $G^*$ is not homeomorphic to $G$.
Another family $h(m, n, s)$ of geodetic graphs of arbitrary diameter~has been presented by Bosak in~\cite{bosak1978geodetic}.
Nevertheless, these graphs are homeomorphic to diameter-three graphs.
The only known infinite families of geodetic graphs with larger diameter and without degree-two vertices are of diameter 4 and 5~\cite{Bridgland83,SrinivasanOA87,EtgarN24}.

\enlargethispage{\baselineskip}
A different approach than finding infinite families of geodetic graphs is to search for \emph{all} geodetic graphs up to a certain number of vertices.
Very little research has been done in this direction.
Still, as part of an undergraduate research project, John Cu computed the biconnected geodetic graphs with up to 11 vertices~\cite{cu2021length} (see also A337179 and A337178 in OEIS~\cite{oeis}).
However, it is difficult to generalize this approach to larger graphs due to the fact that the number of graphs on $n$ vertices grows like $2^{\Theta(n^2)}$.

\section{Preliminaries}
Throughout this paper we consider only undirected finite simple graphs.
We assume the reader is familiar with the basics of graph theory and use notation similar to~\cite{bollobas2013modern}.
We denote the distance between vertices $u$ and $v$ by $d(u,v)$.
A \emph{geodesic} from $u$ to $v$ is a shortest path.
A graph $G$ is called \emph{geodetic} if for all $u,v \in V$ there is at most one geodesic from $u$ to $v$.
A graph $G$ is \emph{biconnected} if $G$ is connected and for every $v \in V$ the induced subgraph of $V \setminus \{v\}$ is connected.
A graph is geodetic if and only if every maximal biconnected subgraph (block) is geodetic~\cite{stemple1968planar}.
We introduce some basic lemmas used in our search algorithm as well as in the constructions below.

\begin{lemma}[Theorem 3.3 in~\cite{stemple1974geodetic}]
  \label{lem:diamonds}
  Let $G = (V, E)$ be a geodetic graph.
  If four vertices $u, v, w, x \in V$ form a 4-cycle, then they induce a complete subgraph.
\end{lemma}

\begin{lemma}[Theorem 3.5 in~\cite{stemple1974geodetic}]
  \label{lem:adjacent_part_of_clique}
  Let $G = (V, E)$ be a geodetic graph, $C \subseteq V$ a clique in $G$ and $v$ adjacent to at least two distinct vertices in $C$.
  Then $C \cup \{ v \}$ is a clique.
\end{lemma}

\noindent
We denote by $\cN^r(v) = \{u \mid d(v, u) = r\}$ the set of vertices at distance $r$ from $v$.
A consequence of the previous two lemmas is the following lemma, describing the neighbourhood of a vertex in a geodetic graph.

\begin{lemma}
  \label{lem:partition_neighbours_cliques}
  Let $G = (V, E)$ be a geodetic graph and $v \in V$ a vertex.
  Then $\cN^1(v)$ (the neighbours of $v$) can be partitioned into a set of disjoint cliques.
\end{lemma}

\begin{lemma}[Unique Predecessor Theorem~\cite{parthasarathy1982some}]
  \label{lem:unique_predessesor}
  The graph $G = (V, E)$ with diameter $d$ is geodetic if and only if for every $v \in V$ each point of $\cN^{r}(v)$ has a unique neighbour in $\cN^{r-1}(v)$ for each $r$ with $2 \le r \le d$.
\end{lemma}

\noindent
\cref{lem:unique_predessesor} leads to the notion of a shortest-path tree:
Let $G$ be a connected geodetic graph and $r \in V$.
Let $T \subseteq E$ be the set of all edges occurring on some shortest path from $r$ to another vertex (\ie we take the union of all shortest paths starting at $r$).
From \Cref{lem:unique_predessesor} it follows that the graph $(V,T)$ is a tree~-- the \emph{shortest-path tree} rooted at $r$ (we also call simply $T$ the shortest-path tree when there is no risk of confusion).
We can assign a \emph{level} to each vertex $v \in V$, which is given by its distance from the root $r$.
Another consequence of \Cref{lem:unique_predessesor} is that, other than $T$, the edge set of $G$ contains only edges between vertices on the same level of $T$.

\begin{lemma}
  \label{lem:forbidden-edges}
  Let $G$ be a geodetic graph with a shortest-path tree $T$.
  If there is an edge $\{u, v\}$ between vertices on the same level, then for no $u'$ in the subtree rooted at $u$ and $v'$ in the subtree  rooted at  $v$ there is an edge $\{u', v'\}$.
\end{lemma}

\begin{proof}
	First, observe that, if $u'$ and $v'$ are on different levels of $T$, an edge $\{u', v'\}$ would contradict the fact that $T$ is a shortest-path tree of a geodetic graph.
    Now, let $d$ be the length of the path $u - u'$ in the shortest-path tree.
    Then $d$ is also the length of the path $v - v'$ in the shortest-path tree.
    Now there are two paths from $u$ to $v'$ of length $d + 1$.
    Since the shortest path must be unique, it can have length at most $d$, but that is not possible without introducing at least one additional edge between two vertices on different levels of the shortest-path tree.\qed
\end{proof}

A \emph{cut} $C = (S, T)$ is a partition of $V$ into two subsets $S$ and $T$.
The corresponding \emph{cut set} is the subset of edges with one vertex in $S$ and the other one in $T$.
The following proposition is an important special case of the construction presented in~\cite{hic1985construction}.
Indeed, it covers exactly the case that~\cite{hic1985construction} leads to a subdivision of the original graph.
For an example of a corresponding geodetic subdivision, see~\Cref{subsec:joining-two-odd-cycles}.

\begin{restatable}{proposition}{propsubdivcut}
  \label{prop:subdivcut}
  Let $G = (V, E)$ be a geodetic graph and $C=(S,T)$ be a cut of $G$ with the following properties:
  \vspace{-2mm}
  \begin{enumerate}
    \item $S$ and $T$ are geodetically closed (meaning that for all $u,v\in S$ the shortest path from $u$ to $v$ does not leave $S$).
    \item\label{subdivcutCondII} Given two edges $\{u, v\}$ and $\{w, x\}$ from the cut set with $u,w \in S$, $v,x \in T$, then $d(u, w) + d(v, x)$ is odd.
    \vspace{-2mm}
  \end{enumerate}
  Then, for every $k$, subdividing the edges in the cut set by adding $k$ vertices to each results in a geodetic graph $G'$.
\end{restatable}

By~\cite[Corollary 1]{GorovoyZ22}, every biconnected non-complete geodetic graph containing a clique of size $m$ also contains $K_{1,m}$ (complete bipartite graph with partitions of size $1$ and $m$) as induced subgraph.
Here, we need a similar but slightly different statement, which can be shown by modifying the proof of~\cite[Corollary 1]{GorovoyZ22}.

\begin{restatable}{proposition}{cliqueindepset}
  \label{lem:clique_independentset_bicon}
  Let $G = (V, E)$ be a biconnected geodetic graph containing a clique $C$.
  Then $G$ is either complete or every vertex $u \in C$ has a neighbour $x_u$ outside of $C$ such that $\{ x_u \mid u \in C \}$ is an independent set of size $|C|$.
\end{restatable}

\section{Enumerating Geodetic Graphs Using \texttt{nauty}}

John Cu computed the biconnected geodetic graphs with up to 11 vertices as part of his undergraduate research project~\cite{cu2021length}.
 In this section, we reproduce his results by a computer search using \texttt{nauty}, a popular graph isomorphism library by B.~McKay~\cite{mckay2014practical}.
Here we rely on \texttt{nauty}'s \texttt{geng} utility: we generate all biconnected graphs using \cref{lem:diamonds} for pruning and then check for each one whether it is geodetic.
We have written a simple function to check whether a graph is geodetic shown in \Cref{lst:check_geodetic} in the appendix.
That function can be baked right into \texttt{nauty}'s \texttt{geng} tool by compiling it with \texttt{-DPRUNE=check\_geodetic}.

We run \texttt{geng} with the \texttt{-C} argument to only generate biconnected graphs.
Spreading the workload to 24 threads, we were able to enumerate all biconnected geodetic graphs with up to 16 vertices in a couple of days.
For $n=16$, there were approximately $0.14\cdot 10^{12}$ intermediate and $1.65\cdot 10^{12}$ final graphs for which the prune function was called.
These numbers suggest that this approach is infeasible for larger values of $n$.

As with any proof relying on exhaustive computer search, there is a possibility that we might have missed some geodetic graph due to a coding or hardware error.
Using \texttt{nauty} gives us confidence in our results.
We wrote only 40 lines of code ourselves, which is still feasible to check for correctness.
Additionally we rely on the correctness of the \texttt{nauty} package.

\section{A Custom Search Algorithm for Enumerating Geodetic Graphs}
\label{subsec:custom-search-algorithm}

To efficiently enumerate geodetic graphs with more than $16$ vertices, we need a different approach than above.
Instead of enumerating all graphs and checking whether they are geodetic, we start with a tree.
The idea is that the tree could be a shortest-path tree of a geodetic graph rooted at a vertex $r$ of maximum degree.
If it is, we can obtain that graph by adding more edges to the tree.
There are far fewer trees than graphs, so enumerating all trees is more feasible.
We use the algorithm from~\cite{li1999advantages} to generate all rooted trees.
Next, we describe our backtracking algorithm for extending a tree to a geodetic graph.

\vspace{-3mm}
\subsubsection*{Overview.}
To extend a tree to a geodetic graph, we implement a backtracking search to add additional edges.
The search traverses the vertices of the tree in level order.
By \cref{lem:forbidden-edges}, all non shortest-path tree edges must be between vertices on the same level, so it is sufficient to consider those.
Among the edges between vertices on the same level, we distinguish sibling edges and non-sibling edges.
A \emph{sibling edge} is an edge between two vertices that share the same parent.
We call vertices sharing the same parent \emph{siblings}.
As a consequence of this distinction, the search alternates between two steps, each being applied to an entire level of the tree: partitioning sibling vertices into cliques and generating non-sibling edges.
\Cref{fig:example-search-space-small} illustrates the two steps.
At any point of time, each vertex has one of the following three states~-- with the stated invariants for the subsequent recursive steps.
\vspace{-2mm}
\begin{description}
  \item[unprocessed] Additional incident edges may still be added.
  \item[semi-complete] Only non-sibling edges can still be added, incident sibling edges will not change anymore.
  \item[complete] No additional incident edges can be added.
\end{description}

\vspace{-5mm}
\paragraph*{Step 1: Partition sibling vertices into cliques.}
This step is motivated by \cref{lem:partition_neighbours_cliques}, which states that the neighbours of each vertex in a geodetic graph are partitioned into disjoint cliques.
The algorithm recursively goes through all sets of siblings on the level.
For a given set of siblings, each partition is checked by adding the edges from the partition and continuing recursively with the next set of siblings.
After fixing a partition, the vertices go from the unprocessed to the semi-complete state.

\vspace{-2mm}
\paragraph*{Step 2: Generate non-sibling edges.}
Once the first step is complete for all vertices on a level, the algorithm recursively traverses the vertices of that same level to generate the non-sibling edges.
For the traversal we assume an order among vertices on the same level.
For each vertex, the algorithm recurses on all combinations of possible incident edges on the same level.
That is, for each combination, add the edges, mark the vertex as completed and recursively continue the search with the next vertex in level-order.
At that stage only non-sibling edges that connect to vertices not yet in the complete state are considered possible.
Note that, since the first step has been completed for the entire level, every vertex on the level is either in the semi-complete or complete state.
Moreover, the vertices in complete state precede the vertices in the semi-complete state.

\begin{figure}[t]
  \centering
  \input{graphs/example-search-space-small}
  \vspace{-1mm}\caption{\small Excerpt of the search space. In the first depicted step sibling edges for the three vertices on the bottom right are generated. Note that the two other ways to insert a single sibling edge yield graphs isomorphic to the middle one. In the second step non-sibling neighbours of the vertex on the bottom left are generated. A vertex in the completed state is represented by a filled circle, semi-complete by a half-full circle.\vspace{-2mm}}
  \label{fig:example-search-space-small}
\end{figure}

\enlargethispage{\baselineskip}
\vspace{-3mm}
\subsubsection{Implementation Details and Optimizations.}
For generating all trees we use the algorithm from~\cite{li1999advantages}.
We have implemented the algorithm in Rust.
A link to the source code can be found in \cref{sec:source-code}.
In \cref{app:running_times} we present running times of our algorithm and the effect of the different pruning techniques.
We take advantage of parallelization by distributing the trees to different threads.
We store graphs as adjacency matrices using a bit vector (in our case each line is stored in a single 32-bit integer).
In particular, we use the same format as \texttt{nauty}, so that we can perform isomorphism checks efficiently without converting the format.
We use the following pruning techniques to reduce the search space.

\vspace{-2mm}
\paragraph*{Maximum Vertex Degree.}
Recall that we assume the root of the tree to be a vertex of maximum degree in the geodetic graph.
Thus, the degree of the root is an upper bound for the degree of any other vertex.
To take advantage of this, when generating sibling edges, we immediately discard partitions that violate the maximum degree of some vertex.
When generating non-sibling edges, we use this to limit the number of edges added to a vertex in a recursion step and we do not add edges to another vertex that already has the maximum degree.

\vspace{-2mm}
\paragraph*{Non-Clique Neighbour.}
By \Cref{lem:clique_independentset_bicon}, every clique vertex has a neighbour that is not part of the clique.
Thus, when generating the non-sibling edges of a vertex that is part of a clique with its parent and some of its siblings and has no other edges, we make sure to add at least one edge.

\vspace{-2mm}
\paragraph*{Forbidden Edges.}
\Cref{lem:forbidden-edges} states that, whenever there is an edge between two vertices on the same level, there cannot be certain edges on the levels below.
To take advantage of this, we maintain a set of forbidden edges in our implementation.
The forbidden edges are not considered when generating the non-sibling edges. 

\vspace{-2mm}
\paragraph*{Isomorphism Tests.}
The tree and the generated graphs often have large automorphism groups.
Generating the same graph multiple times has a significant impact on the performance.
In our implementation we use two strategies to avoid generating certain isomorphic graphs in the first place:
First, when generating sibling edges, we only consider non-isomorphic partitions (with respect to isomorphism classes of the corresponding subtrees).
Second, when generating non-sibling edges we assume that vertices that are children of the same parent and have isomorphic subtrees are ordered by degree.
Since these only prevent a fraction of isomorphic graphs being generated, we additionally implemented explicit isomorphism checks using \texttt{nauty}.
We added those checks after each recursive generation step.
These are realized using a cache of fixed size, which is the only part of the algorithm that requires a substantial amount of memory.
For each step of the generation process we use a separate cache each containing $5000000$ entries.
In total one thread uses up to 40 GB.

\vspace{-2mm}
\paragraph*{Biconnectivity.}
As we are only interested in biconnected graphs, one pruning step is to test whether the final graph can still be made biconnected.
To do so, we take the set of edges that has already been generated and add all the edges that could still be generated (excluding those in the forbidden edge set) and test whether that graph is biconnected.

\vspace{-2mm}
\paragraph*{Cycles of length 4 and 6.}
We check whether cycles of length 4 and 6 conform to one of the configurations possible in a geodetic graph.
For cycles of length 4 these are derived from \Cref{lem:diamonds}.
The possible configurations in which a 6-cycle can appear have been classified by Stemple~\cite{stemple1974geodetic}.
We describe the pruning step in more detail in~\Cref{app:small-cycles}.

\section{Experimental Results}
\enlargethispage{\baselineskip}

\begin{table}[t]
  \centering
  \footnotesize
  \caption{\small Number of geodetic graphs.}
  \label{tab:number_geodetic_graphs}

  \setlength{\multicolsep}{0pt}
  \setlength{\columnsep}{0pt}
    \setlength{\tabcolsep}{2pt}
  \begin{multicols}{2}
    \begin{tabular}{rrrr}
      \toprule
      \footnotesize \textbf{Vertices} & \footnotesize\textbf{Total} & \footnotesize\textbf{Connected} & \footnotesize\textbf{Bicon.} \\
      \midrule
   \footnotesize    1                 & \footnotesize 1              & \footnotesize 1                  & \footnotesize --                    \\
   \footnotesize    2                 & \footnotesize 2              & \footnotesize 1                  & \footnotesize 1                    \\
   \footnotesize    3                 & \footnotesize 4              & \footnotesize 2                  & \footnotesize 1                    \\
   \footnotesize    4                 & \footnotesize 9              & \footnotesize 4                  & \footnotesize 1                    \\
   \footnotesize    5                 & \footnotesize 21             & \footnotesize 10                 & \footnotesize 2                    \\
   \footnotesize    6                 & \footnotesize 52             & \footnotesize 23                 & \footnotesize 1                    \\
   \footnotesize    7                 & \footnotesize 138            & \footnotesize 66                 & \footnotesize 3                    \\
   \footnotesize    8                 & \footnotesize 384            & \footnotesize 185                & \footnotesize 1                    \\
   \footnotesize    9                 & \footnotesize 1146           & \footnotesize 586                & \footnotesize 3                    \\
   \footnotesize    10                & \footnotesize 3563           & \footnotesize 1880               & \footnotesize 4                    \\
   \footnotesize    11                & \footnotesize 11599          & \footnotesize 6360               & \footnotesize 3                    \\
   \footnotesize    12                & \footnotesize 39013          & \footnotesize 21975              & \footnotesize 1                    \\
   \footnotesize    13                & \footnotesize 135359         & \footnotesize 78230              & \footnotesize 9                    \\
    \end{tabular}
    \columnbreak

    \begin{tabular}{rrrr}
      \toprule
      \footnotesize\textbf{Vertices} & \footnotesize\textbf{Total} & \footnotesize\textbf{Connected} & \footnotesize\textbf{Bicon.} \\
      \midrule
\footnotesize      14                &\footnotesize 480489         &\footnotesize 283087             &\footnotesize 2                    \\
\footnotesize      15                &\footnotesize 1741210        &\footnotesize 1043329            &\footnotesize 4                    \\
\footnotesize      16                &\footnotesize 6413809        &\footnotesize 3895505            &\footnotesize 8                    \\
\footnotesize      17                &\footnotesize 23966313       &\footnotesize 14726263           &\footnotesize 6                    \\
\footnotesize      18                &\footnotesize 90633268       &\footnotesize 56234210           &\footnotesize 5                    \\
\footnotesize      19                &\footnotesize 346364720      &\footnotesize 216719056          &\footnotesize 13                   \\
\footnotesize      20                &\footnotesize 1335756420     &\footnotesize 841857211          &\footnotesize 3                    \\
\footnotesize      21                &\footnotesize 5192999602     &\footnotesize 3293753840         &\footnotesize 13                   \\
\footnotesize      22                &\footnotesize 20333304508    &\footnotesize 12969219563        &\footnotesize 19                   \\
\footnotesize      23                &\footnotesize                &\footnotesize                    &\footnotesize 11                   \\
\footnotesize      24                &\footnotesize                &\footnotesize                    &\footnotesize 3                    \\
\footnotesize      25                &\footnotesize                &\footnotesize                    &\footnotesize 33                   \\
      \bottomrule
    \end{tabular}
  \end{multicols}
\end{table}

Using our algorithm we enumerated all biconnected geodetic graphs with $n$ vertices for $n \le 25$. This took 16 days of computation time of which almost 15 days were spent for $n=25$ using 24 threads.
Thus, we are able to confirm the results Cu and Elder obtained for $n \le 11$~\cite{cu2021length}.
\Cref{tab:number_geodetic_graphs} shows the number of geodetic graphs with $n$ vertices (we did not compute the numbers of connected geodetic graphs for $n \geq 23$ because there are too many of them).
The number of connected geodetic graph is listed as A337179 in OEIS, the number of biconnected geodetic graphs as A337178~\cite{oeis}.
A list of all the biconnected graphs discovered by our computer search is given in \Cref{tab:graphs}.
There, the graphs $G_c(p)$, $G_d(p)$ and $G_e(p)$ denote the constructions given in~\cite{blokhuis1988geodetic} in Sections 2c--2e and $p$ is the order of the projective/affine plane the construction is based on.
Unless noted otherwise, subdivision refers to \Cref{prop:subdivcut}.
The columns describe properties of the graphs:
$r$ denotes the radius ($\min_{u\in V}\max_{v\in V} d(u, v) $), $d$ the diameter and $\delta$ the minimum vertex degree.
The ``re.'' and ``ha.'' columns state whether the graph is regular (resp.~hamiltonian).
By $\operatorname{Aut}$ we denote the automorphism group of the graph using the following notation:
$C_n$ is the cyclic group of order $n$, $D_{2n}$ the dihedral group of order $2n$ and $S_n$ the symmetric group on $n$ elements.
A semidirect product is denoted by $\rtimes$.

As expected, our search found all the known geodetic biconnected graphs with up to 25 vertices -- that is complete graphs and the graphs yielded by the constructions from~\cite{blokhuis1988geodetic,bosak1978geodetic,parthasarathy1982some}.
Most of the graphs found by our search are subdivisions of other geodetic graphs (in particular, complete graphs).
Moreover, as foreseen by \cite{plesnik1977two,stemple1979geodetic}, there are certain values $n$ for which we found many such subdivisions and other values $n$ that do not admit any such subdivision such as $n=12,14$, or $24$ (moreover, for $n = 18,20$ there are also no subdivisions of complete graphs but there are subdivisions of other smaller geodetic graphs).
The most exciting discovery are two graphs, $F_5$ and $H(2,2,2,0)$, which to the best of our knowledge have not appeared in the literature before nor can be build using one of the previously known constructions.
We will use these in the next section, two derive two new families of geodetic graphs.
Finally, note here that, except the complete graphs, the Petersen graph and four graphs described in~\cite{blokhuis1988geodetic}, all graphs have minimum vertex degree two.
Moreover, the only regular graphs in the table are the Petersen graph, cycles and complete graphs.

\begin{table}[t]
	\caption{\small List of biconnected geodetic graphs.}
	\label{tab:graphs}
	\begin{minipage}{\textwidth}
		\centering
		\enlargethispage{2\baselineskip}
		\small
		\setlength{\multicolsep}{0pt}
    \setlength{\columnsep}{0pt}
		\setlength{\tabcolsep}{1.0pt}
		\begin{multicols}{2}
			\begin{tabular}{rp{62pt}rrrp{9pt}p{9pt}p{40pt}}
				\toprule
				$n$ & \textbf{Graph} & \textbf{$r$} & $d$ & $\delta$ & re. & ha. & $\operatorname{Aut}$\\
				\midrule
				$2$ & $K_2$ & $1$ & $1$ & $1$ & \marky &  & $C_{2}$\\\hline
				$3$ & $C_3$ & $1$ & $1$ & $2$ & \marky & \marky & $S_{3}$\\\hline
				$4$ & $K_4$ & $1$ & $1$ & $3$ & \marky & \marky & $S_{4}$\\\hline
				\multirow{2}{*}{$5$} & $C_5$ & $2$ & $2$ & $2$ & \marky & \marky & $D_{10}$\\
				& $K_5$ & $1$ & $1$ & $4$ & \marky & \marky & $S_{5}$\\\hline
				$6$ & $K_6$ & $1$ & $1$ & $5$ & \marky & \marky & $S_{6}$\\\hline
				\multirow{3}{*}{$7$} & $C_7$ & $3$ & $3$ & $2$ & \marky & \marky & $D_{14}$\\
				& Subdiv.\ of $K_4$ & $2$ & $2$ & $2$ &  &  & $S_{3}$\\
				& $K_7$ & $1$ & $1$ & $6$ & \marky & \marky & $S_{7}$\\\hline
				$8$ & $K_8$ & $1$ & $1$ & $7$ & \marky & \marky & $S_{8}$\\\hline
				\multirow{3}{*}{$9$} & $C_9$ & $4$ & $4$ & $2$ & \marky & \marky & $D_{18}$\\
				& Subdiv.\ of $K_5$ & $2$ & $2$ & $2$ &  &  & $S_{4}$\\
				& $K_9$ & $1$ & $1$ & $8$ & \marky & \marky & $S_{9}$\\\hline
				\multirow{3}{*}{$10$} & \multicolumn{7}{l}{2 Subdiv.\ of $K_4$}  \\
				& Petersen Graph & $2$ & $2$ & $3$ & \marky &  & $S_{5}$\\
				& $K_{10}$ & $1$ & $1$ & $9$ & \marky & \marky & $S_{10}$\\\hline
				\multirow{3}{*}{$11$} & $C_{11}$ & $5$ & $5$ & $2$ & \marky & \marky & $D_{22}$\\
				& Subdiv.\ of $K_6$ & $2$ & $2$ & $2$ &  &  & $S_{5}$\\
				& $K_{11}$ & $1$ & $1$ & $10$ & \marky & \marky & $S_{11}$\\\hline
				$12$ & $K_{12}$ & $1$ & $1$ & $11$ & \marky & \marky & $S_{12}$\\\hline
				\multirow{6}{*}{$13$} & $C_{13}$ & $6$ & $6$ & $2$ & \marky & \marky & $D_{26}$\\
				& \multicolumn{7}{l}{3 Subdiv.\ of $K_4$}  \\
				& \multicolumn{7}{l}{2 Subdiv.\ of $K_5$}  \\
				& Subdiv.\ of $K_7$ & $2$ & $2$ & $2$ &  &  & $S_{6}$\\
				& $G_d(2) = G_e(3)$ & $2$ & $2$ & $3$ &  & \marky & $S_{4}$\\
				& $K_{13}$ & $1$ & $1$ & $12$ & \marky & \marky & $S_{13}$\\\hline
				\multirow{2}{*}{$14$} & $h(3, 2, 0)$ & $3$ & $3$ & $2$ &  &  & $D_{12}$\\
				& $K_{14}$ & $1$ & $1$ & $13$ & \marky & \marky & $S_{14}$\\\hline
				\multirow{4}{*}{$15$} & $C_{15}$ & $7$ & $7$ & $2$ & \marky & \marky & $D_{30}$\\
				& Subdiv.\ of $K_8$ & $2$ & $2$ & $2$ &  &  & $S_{7}$\\
				& $\WP_3$ & $3$ & $3$ & $2$ &  &  & $C_{5} \rtimes C_{4}$\\
				& $K_{15}$ & $1$ & $1$ & $14$ & \marky & \marky & $S_{15}$\\\hline
				\multirow{3}{*}{$16$} & \multicolumn{7}{l}{5 Subdiv.\ of $K_4$}  \\
				& \multicolumn{7}{l}{2 Subdiv.\ of $K_6$}  \\
				& $K_{16}$ & $1$ & $1$ & $15$ & \marky & \marky & $S_{16}$\\\hline
				\multirow{4}{*}{$17$} & $C_{17}$ & $8$ & $8$ & $2$ & \marky & \marky & $D_{34}$\\
				& \multicolumn{7}{l}{3 Subdiv.\ of $K_{5}$}  \\
				& Subdiv.\ of $K_9$ & $2$ & $2$ & $2$ &  &  & $S_{8}$\\
				& $K_{17}$ & $1$ & $1$ & $16$ & \marky & \marky & $S_{17}$\\\hline
				\multirow{4}{*}{$18$} & \multicolumn{7}{l}{2 Subdiv.\ of $h(3, 2, 0)$}  \\
				& $h(4, 2, 0)$ & $3$ & $3$ & $2$ &  &  & $C_{2} \times S_{4}$\\
				& $h(3, 3, 0)$ & $3$ & $3$ & $2$ &  &  & $S_{3} \times S_{3}$\\
				& $K_{18}$ & $1$ & $1$ & $17$ & \marky & \marky & $S_{18}$\\\hline
				\multirow{4}{*}{$19$} & $C_{19}$ & $9$ & $9$ & $2$ & \marky & \marky & $D_{38}$\\
				& \multicolumn{7}{l}{6 Subdiv.\ of $K_4$} \\
				& \multicolumn{7}{l}{2 Subdiv.\ of $K_7$} \\
				& Subdiv.\ of $K_{10}$ & $2$ & $2$ & $2$ &  &  & $S_{9}$\\[1mm]
			\end{tabular}

			\begin{tabular}{rp{62pt}rrrp{9pt}p{9pt}p{40pt}}
				\toprule
				$n$ & \textbf{Graph} & \textbf{$r$} & $d$ & $\delta$ & re. & ha. & $\operatorname{Aut}$\\
				\midrule
        \multirow{3}{*}{$19$} & Subdiv.\ of $G_d(2)$ & $3$ & $3$ & $2$ &  &  & $S_{4}$\\
        & $F_3 = K_5^*$ & $4$ & $5$ & $2$ &  &  & $D_{8}$\\
        & $K_{19}$ & $1$ & $1$ & $18$ & \marky & \marky & $S_{19}$\\\hline
				\multirow{4}{*}{$20$} & Subdiv.\ of Petersen Graph\footnote{\label{fn:subdiv}Not explained by \cref{prop:subdivcut}.} & $4$ & $4$ & $2$ &  &  & $D_{8}$\\
				& $\WP_4$ & $4$ & $4$ & $2$ &  &  & $C_{5} \rtimes C_{4}$\\
				& $K_{20}$ & $1$ & $1$ & $19$ & \marky & \marky & $S_{20}$\\\hline
				\multirow{9}{*}{$21$} & $C_{21}$ & $10$ & $10$ & $2$ & \marky & \marky & $D_{42}$\\
				& 5 {subdiv.} of $K_5$ &  \\
				& \multicolumn{7}{l}{3 Subdiv.\ of $K_6$}  \\
				& Subdiv.\ of $K_{11}$ & $2$ & $2$ & $2$ &  &  & $S_{10}$\\
				& $G_c(3)$ & $2$ & $2$ & $4$ &  & \marky & $(C_3 \!\times\! C_3) \rtimes \mathrm{GL}(2,3)$ \\
				& $G_e(4)$ & $2$ & $2$ & $4$ &  & \marky & $S_{5}$\\
				& $K_{21}$ & $1$ & $1$ & $20$ & \marky & \marky & $S_{21}$\\\hline
				\multirow{6}{*}{$22$} & \multicolumn{7}{l}{9 Subdiv.\ of $K_4$}  \\
				& \multicolumn{7}{l}{2 Subdiv.\ of $K_8$}  \\
				& \multicolumn{7}{l}{5 Subdiv.\ of $h(3, 2, 0)$}  \\
				& $h(3, 4, 0)$ & $3$ & $3$ & $2$ &  &  & $S_{3} \times S_{4}$\\
				& $h(5, 2, 0)$ & $3$ & $3$ & $2$ &  &  & $C_{2} \times S_{5}$\\
				& $K_{22}$ & $1$ & $1$ & $21$ & \marky & \marky & $S_{22}$\\\hline
				\multirow{9}{*}{$23$} & $C_{23}$ & $11$ & $11$ & $2$ & \marky & \marky & $D_{46}$\\
				& Subdiv.\ of $K_{12}$ & $2$ & $2$ & $2$ &  &  & $S_{11}$\\
				& \multicolumn{7}{l}{2 Subdiv.\ of $h(4, 2, 0)$}  \\
				& \multicolumn{7}{l}{2 Subdiv.\ of $h(3, 3, 0)$}  \\
				& Subdiv.\ of $F_3$ & $5$ & $6$ & $2$ &  &  & $C_{2}$\\
				& $h(4, 3, 0)$ & $3$ & $3$ & $2$ &  &  & $S_{3} \times S_{4}$\\
				& $h(3, 2, 1)$ & $5$ & $5$ & $2$ &  &  & $D_{12}$\\
				& $F_5$ & $5$ & $5$ & $2$ &  &  & $D_{8}$\\
				& $K_{23}$ & $1$ & $1$ & $22$ & \marky & \marky & $S_{23}$\\\hline
				\multirow{3}{*}{$24$} & $H(2,2,2,0)$ & $4$ & $4$ & $2$ &  &  & $C_{2} \times A_{4}$ \\
				& $K_6^*$~\cite{parthasarathy1982some}\footnote{\label{fn:star}The construction also depends on a so-called $\mathcal{O}$-cover of $K_6$~-- therefore, different graphs can arise as $K_6^*$.} & $4$ & $5$ & $2$ &  &  & $S_4$ \\
				& $K_{24}$ & $1$ & $1$ & $23$ & \marky & \marky & $S_{24}$\\\hline
				\multirow{14}{*}{$25$} & $C_{25}$ & $12$ & $12$ & $2$ & \marky & \marky & $D_{50}$ \\
				& 11 Subdiv.\ of $K_4$ & \\
				& 7 Subdiv.\ of $K_5$ & \\
				& 3 Subdiv.\ of $K_7$ & \\
				& 2 Subdiv.\ of $K_9$ & \\
				& Subdiv.\ of $K_{13}$ & $2$ & $2$ & $2$ &  &  & $S_{12}$ \\
				& \multicolumn{7}{l}{2 Subdiv.\ of Petersen Graph\footnoteref{fn:subdiv}} \\
				& Subdiv.\ of $G_d(2)$ & $4$ & $4$ & $2$ &  &  & $S_4$ \\
				& $\WP_5$ & $5$ & $5$ & $2$ &  &  & $C_5 \rtimes C_4$ \\
				& $K_6^*$~\cite{parthasarathy1982some}\footnoteref{fn:star} & $4$ & $5$ & $2$ &  &  & $D_{12}$ \\
				& Pulling\footnote{The pulling operation has been introduced by Plesnik~\cite{plesnik1984construction}. See \Cref{sec:pull}.} of\newline $K_6(1,1,1,1,1,0)$ & 3 & 5 & 2 &  &  & $D_8$ \\
				& $G_d(3)$ & $2$ & $2$ & $4$ &  & \marky & $(C_3 \!\times\! C_3) \rtimes \mathrm{GL}(2,3)$ \\
				& $K_{25}$ & $1$ & $1$ & $24$ & \marky & \marky & $S_{25}$ \\
				\bottomrule
			\end{tabular}
		\end{multicols}
	\end{minipage}
\end{table}

\vspace{-2mm}
\paragraph*{Regular geodetic graphs.}
Odd cycles and complete graphs are trivially geodetic.
Beyond those there are the strongly regular Moore graphs of diameter two:
the Petersen graph, the Hoffman-Singleton graph and possibly a graph of order 3250~\cite{hoffman1960moore}.
In addition, a cubic block of order 28 is presented in~\cite{bosak1978geodetic}.
We used our algorithm to confirm that there are no further regular geodetic graphs with at most 32 vertices.

\section{New Constructions of Geodetic Graphs}\label{subsec:joining-two-odd-cycles}
\enlargethispage{\baselineskip}

In this section, we present two new infinite families of geodetic graphs discovered with the help of our computer search.

\vspace{-2mm}
\paragraph*{Joining Three Cliques.}\label{sec:joining_cliques}

Before describing our first family $H(m, n, p, s)$ of geodetic graphs, we look at the family $h(m, n, s)$ described in~\cite{bosak1978geodetic}.
It will be the starting point for our own construction.
The parameters $m$, $n$ and $s$ are integers with $m, n \ge 2$ and $s \ge 0$.
The graph $h(m, n, s)$ is constructed from a $K_m$, a $K_n$ and $m$ copies of a star with $n$ leaves.
The subgraphs are joined by paths of length $s + 1$ as can be seen in \Cref{fig:family-h}.

\begin{figure}[t]
  \centering

  \begin{subfigure}[b]{0.3\textwidth}
    \centering
    \begin{tikzpicture}[style=thick,xscale=0.6,yscale=0.8,rotate=90]

  \draw (0,1.75) node[ellipse, draw, thick, fill=black!10, minimum width=25mm, minimum height=1mm, inner sep=1mm, align=center] {};
  \draw (7,1.75) node[ellipse, draw, thick, fill=black!10, minimum width=25mm, minimum height=1mm, inner sep=1mm, align=center] {};
  \node at (0, -1) {$K_m$};
  \node at (7, -1) {$K_n$};
  \draw [thick,decoration={brace,raise=-0.4cm,mirror},decorate,color=black] (0.8,-1) -- (2.2,-1) node [pos=0.5,yshift=0.0cm,rotate=90] {\small $s$ vertices};
  \draw [thick,decoration={brace,raise=-0.25cm,mirror},decorate,color=black] (4.3,-1) -- (5.7,-1) node [pos=0.5,yshift=0.0cm,rotate=90] {\small $s$ vertices};
  \foreach \y in {0,2,3.5}{
    \fill (0,\y) circle[shift only,radius=2pt];
    \fill (1,\y) circle[shift only,radius=2pt];
    \fill (2,\y) circle[shift only,radius=2pt];
    \fill (3,\y) circle[shift only,radius=2pt];
    \fill (7,\y) circle[shift only,radius=2pt];
    \draw (0,\y) -- (1.2,\y);
    \draw (1.8,\y) -- (3,\y);
    \foreach \yy in {-0.35,0.05,0.35}{
      \draw (3, \y) -- (3.5,\y+\yy);
      \fill (3.5,\y+\yy) circle[shift only,radius=2pt];
      \fill (4.5,\y+\yy) circle[shift only,radius=2pt];
      \fill (5.5,\y+\yy) circle[shift only,radius=2pt];
      \draw (3.5,\y+\yy) -- (4.7,\y+\yy);
      \draw (5.3,\y+\yy) -- (5.5,\y+\yy);
      \draw (5.5,\y+\yy) -- (7,\yy*5+1.75);
      \node[rotate=90] at (5,\y+\yy) {$\cdots$};
    }
    \node[rotate=90] at (4,\y-0.15) {$\cdots$};
    \node[rotate=90] at (5,\y-0.15) {$\cdots$};
    \node[rotate=90] at (1.5,\y) {$\cdots$};
  }
  \node[rotate=90] at (1.5,1) {$\cdots$};
  \node[rotate=90] at (3.25,1) {$\cdots$};
\end{tikzpicture}
    \caption{\small $h(m, n, s)$}
    \label{fig:family-h}
  \end{subfigure}
  \hfill
  \begin{subfigure}[b]{0.68\textwidth}
    \centering
    \begin{tikzpicture}[style=thick,scale=0.4]

  \begin{scope}[rotate=60]
    \draw (0,5.8) node[transform shape, ellipse, draw, thick, fill=black!10, minimum width=5mm, minimum height=40mm, inner sep=1mm, align=center] {};
    \node at (0, 8.3) {$K_m$};
  \end{scope}

  \begin{scope}[rotate=180]
    \draw (0,5.8) node[transform shape, ellipse, draw, thick, fill=black!10, minimum width=5mm, minimum height=40mm, inner sep=1mm, align=center] {};
    \node at (0, 8.3) {$K_p$};
  \end{scope}

  \begin{scope}[rotate=-60]
    \draw (0,5.8) node[transform shape, ellipse, draw, thick, fill=black!10, minimum width=5mm, minimum height=40mm, inner sep=1mm, align=center] {};
    \node at (0, 8.3) {$K_n$};
  \end{scope}

  \begin{scope}[xshift=-3.5cm,yshift=2.03cm]
    \draw [thick,decoration={brace,raise=-0.4cm},decorate,color=black] (0.8,4.8) -- (2.2,4.8) node [pos=0.5,yshift=0.0cm] {\small $s$ vertices};
    \draw [thick,decoration={brace,raise=-0.25cm},decorate,color=black] (4.3,4.8) -- (5.7,4.8) node [pos=0.5,yshift=0.0cm] {\small $s$ vertices};
    \fill (0,0) circle[shift only,radius=2pt];
    \fill (150:2cm) circle[shift only,radius=2pt];
    \fill (150:3.5cm) circle[shift only,radius=2pt];
    \fill ($(7,0)+(30:0cm)$) circle[shift only,radius=2pt];
    \fill ($(7,0)+(30:2cm)$) circle[shift only,radius=2pt];
    \fill ($(7,0)+(30:3.5cm)$) circle[shift only,radius=2pt];
    \foreach \y in {0,2,3.5}{
      \fill (1,\y) circle[shift only,radius=2pt];
      \fill (2,\y) circle[shift only,radius=2pt];
      \fill (3,\y) circle[shift only,radius=2pt];
      \draw (150:\y cm) -- (1,\y);
      \draw (1,\y) -- (1.2,\y);
      \draw (1.8,\y) -- (3,\y);
      \foreach \yy/\zz in {-0.45/-0.35,0.05/0.05,0.45/0.35}{
        \draw (3, \y) -- (3.5,\y+\yy);
        \fill (3.5,\y+\yy) circle[shift only,radius=2pt];
        \fill (4.5,\y+\yy) circle[shift only,radius=2pt];
        \fill (5.5,\y+\yy) circle[shift only,radius=2pt];
        \draw (3.5,\y+\yy) -- (4.7,\y+\yy);
        \draw (5.3,\y+\yy) -- (5.5,\y+\yy);
        \draw (5.5,\y+\yy) -- ($(7,0)+(30:\zz*5cm+1.75cm)$);
        \node[transform shape] at (5,\y+\yy) {$\cdots$};
      }
      \node[transform shape] at (4,\y-0.15) {$\cdots$};
      \node[transform shape] at (5,\y-0.15) {$\cdots$};
      \node[transform shape] at (1.5,\y) {$\cdots$};
    }
    \node[transform shape] at (1.5,1) {$\cdots$};
    \node[transform shape] at (3.25,1) {$\cdots$};
  \end{scope}

  \begin{scope}[rotate=120,xshift=-3.5cm,yshift=2.03cm]
    \draw [thick,decoration={brace,raise=-0.4cm},decorate,color=black] (0.8,4.8) -- (2.2,4.8) node [rotate=-60,pos=0.5,yshift=0.0cm] {\small $s$ vertices};
    \draw [thick,decoration={brace,raise=-0.25cm},decorate,color=black] (4.3,4.8) -- (5.7,4.8) node [rotate=-60,pos=0.5,yshift=0.0cm] {\small $s$ vertices};
    \fill (0,0) circle[shift only,radius=2pt];
    \fill (150:2cm) circle[shift only,radius=2pt];
    \fill (150:3.5cm) circle[shift only,radius=2pt];
    \fill ($(7,0)+(30:0cm)$) circle[shift only,radius=2pt];
    \fill ($(7,0)+(30:2cm)$) circle[shift only,radius=2pt];
    \fill ($(7,0)+(30:3.5cm)$) circle[shift only,radius=2pt];
    \foreach \y in {0,2,3.5}{
      \fill (1,\y) circle[shift only,radius=2pt];
      \fill (2,\y) circle[shift only,radius=2pt];
      \fill (3,\y) circle[shift only,radius=2pt];
      \draw (150:\y cm) -- (1,\y);
      \draw (1,\y) -- (1.2,\y);
      \draw (1.8,\y) -- (3,\y);
      \foreach \yy/\zz in {-0.45/-0.35,0.05/0.05,0.45/0.35}{
        \draw (3, \y) -- (3.5,\y+\yy);
        \fill (3.5,\y+\yy) circle[shift only,radius=2pt];
        \fill (4.5,\y+\yy) circle[shift only,radius=2pt];
        \fill (5.5,\y+\yy) circle[shift only,radius=2pt];
        \draw (3.5,\y+\yy) -- (4.7,\y+\yy);
        \draw (5.3,\y+\yy) -- (5.5,\y+\yy);
        \draw (5.5,\y+\yy) -- ($(7,0)+(30:\zz*5cm+1.75cm)$);
        \node[transform shape] at (5,\y+\yy) {$\cdots$};
      }
      \node[transform shape] at (4,\y-0.15) {$\cdots$};
      \node[transform shape] at (5,\y-0.15) {$\cdots$};
      \node[transform shape] at (1.5,\y) {$\cdots$};
    }
    \node[transform shape] at (1.5,1) {$\cdots$};
    \node[transform shape] at (3.25,1) {$\cdots$};
  \end{scope}

  \begin{scope}[rotate=-120,xshift=-3.5cm,yshift=2.03cm]
    \draw [thick,decoration={brace,raise=-0.4cm},decorate,color=black] (0.8,4.8) -- (2.2,4.8) node [rotate=60,pos=0.5,yshift=0.0cm] {\small $s$ vertices};
    \draw [thick,decoration={brace,raise=-0.25cm},decorate,color=black] (4.3,4.8) -- (5.7,4.8) node [rotate=60,pos=0.5,yshift=0.0cm] {\small $s$ vertices};
    \fill (0,0) circle[shift only,radius=2pt];
    \fill (150:2cm) circle[shift only,radius=2pt];
    \fill (150:3.5cm) circle[shift only,radius=2pt];
    \fill ($(7,0)+(30:0cm)$) circle[shift only,radius=2pt];
    \fill ($(7,0)+(30:2cm)$) circle[shift only,radius=2pt];
    \fill ($(7,0)+(30:3.5cm)$) circle[shift only,radius=2pt];
    \foreach \y in {0,2,3.5}{
      \fill (1,\y) circle[shift only,radius=2pt];
      \fill (2,\y) circle[shift only,radius=2pt];
      \fill (3,\y) circle[shift only,radius=2pt];
      \draw (150:\y cm) -- (1,\y);
      \draw (1,\y) -- (1.2,\y);
      \draw (1.8,\y) -- (3,\y);
      \foreach \yy/\zz in {-0.45/-0.35,0.05/0.05,0.45/0.35}{
        \draw (3, \y) -- (3.5,\y+\yy);
        \fill (3.5,\y+\yy) circle[shift only,radius=2pt];
        \fill (4.5,\y+\yy) circle[shift only,radius=2pt];
        \fill (5.5,\y+\yy) circle[shift only,radius=2pt];
        \draw (3.5,\y+\yy) -- (4.7,\y+\yy);
        \draw (5.3,\y+\yy) -- (5.5,\y+\yy);
        \draw (5.5,\y+\yy) -- ($(7,0)+(30:\zz*5cm+1.75cm)$);
        \node[transform shape] at (5,\y+\yy) {$\cdots$};
      }
      \node[transform shape] at (4,\y-0.15) {$\cdots$};
      \node[transform shape] at (5,\y-0.15) {$\cdots$};
      \node[transform shape] at (1.5,\y) {$\cdots$};
    }
    \node[transform shape] at (1.5,1) {$\cdots$};
    \node[transform shape] at (3.25,1) {$\cdots$};
  \end{scope}
\end{tikzpicture}
    \caption{\small $H(m, n, p, s)$}
    \label{fig:family-H}
  \end{subfigure}
  \caption{\small The family $h(m, n, s)$~\cite{bosak1978geodetic} and the newly discovered family $H(m, n, p, s)$.}
  \label{fig:family-H-h}
\end{figure}

The graph $H(m, n, p, s)$ is obtained by taking an $h(m, n, s)$, an $h(n, p, s)$ and an $h(p, m, s)$.
The subgraphs are joined together by identifying the cliques.
For instance, each vertex of the $K_n$ in $h(m, n, s)$ is identified with the corresponding vertex of the $K_n$ in $h(n, p, s)$.
The construction is illustrated in \Cref{fig:family-H}.

\vspace{-2mm}
\paragraph*{Joining Two Odd Cycles.}

For an odd integer $k \ge 3$ we construct the graph $F_k$ by taking two copies of $C_k$, labelling their vertices $u_1, \dots, u_k$, respectively $v_1, \dots, v_k$, and adding
\vspace{-2mm}
\begin{itemize}
	\item a vertex $b$ connected to both $u_{(k+1)/2}$ and $v_{(k+1)/2}$,
	\item a vertex $s_1$ connected to $u_1$ and a vertex $s_2$ connected to $u_k$,
	\item vertices $t_1$ and $t_2$, the first connected to $v_1$, the second to $v_k$ and
	\item four paths of length three connecting both $s_1 $and $s_2$ to both $t_1$ and $t_2$.
\end{itemize}
The graphs $F_3$ and $F_5$ are shown in \Cref{fig:Fk}.
The graph $F_3$ coincides with a graph obtained from a $K_5$ using the star operation presented in~\cite{parthasarathy1982some}.

\begin{figure}[t]
	\centering
	\begin{subfigure}[b]{0.49\textwidth}
		\centering
		\begin{tikzpicture}[style=thick]

  \begin{scope}[xshift=-1.5cm]
    \foreach \x in {-90,30,150} {
      \draw (\x:0.4) -- (\x+120:0.4);
    }
    \fill (30:0.4) circle[shift only,radius=2pt] node[right] {$u_1$};
    \fill (-90:0.4) circle[shift only,radius=2pt] node[below] {$u_2$};
    \fill (150:0.4) circle[shift only,radius=2pt] node[left] {$u_3$};
  \end{scope}

  \begin{scope}[xshift=1.5cm]
    \foreach \x in {-90,30,150} {
      \draw (\x:0.4) -- (\x+120:0.4);
    }
    \fill (30:0.4) circle[shift only,radius=2pt] node[right] {$v_1$};
    \fill (-90:0.4) circle[shift only,radius=2pt] node[below] {$v_2$};
    \fill (150:0.4) circle[shift only,radius=2pt] node[left] {$v_3$};
  \end{scope}

  \coordinate (s1) at (-1.2,1.3);
  \coordinate (s2) at (-1,0.6);
  \coordinate (t2) at (1,0.6);
  \coordinate (t1) at (1.2,1.3);

  \draw (s1) -- (t1);
  \draw (s1) -- (t2);
  \draw (s2) -- (t1);
  \draw (s2) -- (t2);

  \draw (s1) -- ($(-1.5cm,0)+(150:0.4)$);
  \draw (s2) -- ($(-1.5cm,0)+(30:0.4)$);
  \draw (t2) -- ($(1.5cm,0)+(150:0.4)$);
  \draw (t1) -- ($(1.5cm,0)+(30:0.4)$);

  \fill (s1) circle[shift only,radius=2pt] node[above] {$s_1$};
  \fill (s2) circle[shift only,radius=2pt] node[above] {$s_2$};
  \fill (t1) circle[shift only,radius=2pt] node[above] {$t_1$};
  \fill (t2) circle[shift only,radius=2pt] node[above] {$t_2$};

  \fill ($(s1)!0.333!(t1)$) circle[shift only,radius=2pt];
  \fill ($(s1)!0.666!(t1)$) circle[shift only,radius=2pt];
  \fill ($(s1)!0.333!(t2)$) circle[shift only,radius=2pt];
  \fill ($(s1)!0.666!(t2)$) circle[shift only,radius=2pt];
  \fill ($(s2)!0.333!(t1)$) circle[shift only,radius=2pt];
  \fill ($(s2)!0.666!(t1)$) circle[shift only,radius=2pt];
  \fill ($(s2)!0.333!(t2)$) circle[shift only,radius=2pt];
  \fill ($(s2)!0.666!(t2)$) circle[shift only,radius=2pt];

  \fill (0, -0.7) circle[shift only,radius=2pt] node[above] {$b$};
  \draw ($(1.5cm,0)+(-90:0.4)$) -- (0, -0.7) -- ($(-1.5cm,0)+(-90:0.4)$);
\end{tikzpicture}
		\caption{\small $F_3$}
	\end{subfigure}
	\hfill
	\begin{subfigure}[b]{0.49\textwidth}
		\centering
		\begin{tikzpicture}[style=thick]

  \begin{scope}[xshift=-1.5cm]
    \foreach \x in {-18,-90,-162,-234,-306} {
      \draw (\x:0.4) -- (\x+72:0.4);
    }
    \fill (-18:0.4) circle[shift only,radius=2pt] node[right] {$u_2$};
    \fill (-90:0.4) circle[shift only,radius=2pt] node[below] {$u_3$};
    \fill (-162:0.4) circle[shift only,radius=2pt] node[left] {$u_4$};
    \fill (-234:0.4) circle[shift only,radius=2pt] node[left] {$u_5$};
    \fill (-306:0.4) circle[shift only,radius=2pt] node[right] {$u_1$};
  \end{scope}

  \begin{scope}[xshift=1.5cm]
    \foreach \x in {-18,-90,-162,-234,-306} {
      \draw (\x:0.4) -- (\x+72:0.4);
    }
    \fill (-18:0.4) circle[shift only,radius=2pt] node[right] {$v_2$};
    \fill (-90:0.4) circle[shift only,radius=2pt] node[below] {$v_3$};
    \fill (-162:0.4) circle[shift only,radius=2pt] node[left] {$v_4$};
    \fill (-234:0.4) circle[shift only,radius=2pt] node[left] {$v_5$};
    \fill (-306:0.4) circle[shift only,radius=2pt] node[right] {$v_1$};
  \end{scope}

  \draw[dashed,thin] (-2.2,0.09) -- (2.2,0.09);

  \coordinate (s1) at (-1.2,1.3);
  \coordinate (s2) at (-1,0.6);
  \coordinate (t2) at (1,0.6);
  \coordinate (t1) at (1.2,1.3);

  \draw (s1) -- (t1);
  \draw (s1) -- (t2);
  \draw (s2) -- (t1);
  \draw (s2) -- (t2);

  \draw (s1) -- ($(-1.5cm,0)+(-234:0.4)$);
  \draw (s2) -- ($(-1.5cm,0)+(-306:0.4)$);
  \draw (t2) -- ($(1.5cm,0)+(-234:0.4)$);
  \draw (t1) -- ($(1.5cm,0)+(-306:0.4)$);

  \fill (s1) circle[shift only,radius=2pt] node[above] {$s_1$};
  \fill (s2) circle[shift only,radius=2pt] node[above] {$s_2$};
  \fill (t1) circle[shift only,radius=2pt] node[above] {$t_1$};
  \fill (t2) circle[shift only,radius=2pt] node[above] {$t_2$};

  \fill ($(s1)!0.333!(t1)$) circle[shift only,radius=2pt];
  \fill ($(s1)!0.666!(t1)$) circle[shift only,radius=2pt];
  \fill ($(s1)!0.333!(t2)$) circle[shift only,radius=2pt];
  \fill ($(s1)!0.666!(t2)$) circle[shift only,radius=2pt];
  \fill ($(s2)!0.333!(t1)$) circle[shift only,radius=2pt];
  \fill ($(s2)!0.666!(t1)$) circle[shift only,radius=2pt];
  \fill ($(s2)!0.333!(t2)$) circle[shift only,radius=2pt];
  \fill ($(s2)!0.666!(t2)$) circle[shift only,radius=2pt];

  \fill (0, -0.7) circle[shift only,radius=2pt] node[above] {$b$};
  \draw ($(1.5cm,0)+(-90:0.4)$) -- (0, -0.7) -- ($(-1.5cm,0)+(-90:0.4)$);
\end{tikzpicture}
		\caption{\small $F_5$}
	\end{subfigure}
  \caption{\small The family $F_k$. The dashed line on the right represents the cut used to obtain any $F_k$ ($k > 5$) from $F_5$ using \Cref{prop:subdivcut}.\vspace{-2mm}}
  \label{fig:Fk}
\end{figure}

\enlargethispage{\baselineskip}

\section{Conclusion and Open Questions}

We implemented an exhaustive computer search in order to find all geodetic graphs with up to 25 vertices. Moreover, we used our algorithm to find all regular geodetic graphs with up to 32 vertices.
A full classification of geodetic graphs remains the aim for future research. We want to conclude this work by mentioning a few questions for further research in the area.

\vspace{-2mm}
\paragraph*{Growth Rate of Biconnected Geodetic Graphs.}

Let $g(n)$ denote the number of biconnected geodetic graphs on $n$ vertices. Looking at the numbers for $g(n)$ with $n\leq 25$ in \cref{tab:number_geodetic_graphs}, gives rise to a natural question about the growth rate of $g(n)$: it seems to grow slowly but there are points (for $n= 20, 24$) where there are almost no biconnected geodetic graphs.
Yet, for certain values of $n$ there is a super-polynomial lower bound:
\begin{observation}\label{obs:lowerbound}
	$g(k^2) \geq e^{2\sqrt{k}}/14$.
\end{observation}
\begin{proof}
	We only count geodetic subdivisions of $K_k$ with $k^2$ vertices.
  By~\cite{plesnik1977two,stemple1979geodetic} (see \cref{sec:intro}), every such subdivision is described by a multi-set of $k$ non-negative integers $\ell_1, \dots, \ell_{k}$ summing up to $k$.
  The number of such multi-sets is the partition number $P(k)$, which is bounded from below by $e^{2\sqrt{k}}/14$~\cite[Corollary 3.1]{maroti2003elementary}.\qed
\end{proof}

\noindent Let us highlight two questions on the growth rate of $g(n)$:
\vspace{-2mm}
\begin{itemize}
	\item Is there a constant $c$ such that there are infinitely many $n$ with $g(n) \leq c$?
	\item How much can the lower bound from \cref{obs:lowerbound} be improved?
	In particular, is there a constant $c$ such that $g(n) \leq 2^{c \sqrt[4]{n}}$ for all $n$?
\end{itemize}

\vspace{-2mm}
\paragraph*{Subdivisions.}

In our exhaustive search we found several geodetic graphs arising as subdivisions of other geodetic graphs.
Most of them fall in one of two categories using either \cref{prop:subdivcut} or the construction from~\cite{plesnik1977two} starting with a complete graph and assigning numbers to the vertices.
Moreover, there is the subdivision by putting the same even number of vertices on every edge of a geodetic graph.
Apart from those constructions, there are cases of geodetic subdivisions that do not fall in any of these categories:
The graph $h(n, m, 1)$ is a subdivision of $h(n, m, 0)$, and we found two geodetic graphs with 25 vertices and one with 20 vertices that are subdivisions of the Petersen graph (which have been studied previously in~\cite{frasser2020geodetic}).
Thus, the question for characterizing all geodetic subdivisions of a given geodetic graph remains an open problem.
Even more generally, which (arbitrary) graphs have a geodetic subdivision?

\vspace{-2mm}
\paragraph*{Extending the Exhaustive Search.}
Using the exhaustive search for geodetic graphs with more vertices is highly non-trivial.
Given the running times for up to 25 vertices, it seems reasonable that $n=26$ could be finished within a few months.
Larger values for $n$ seem only possible using some new heuristics for the pruning method.
One possible such pruning heuristic would be to use a lemma similar to \cref{lem:stemple6} how an 8-cycle (instead of a 6-cycle) can be embedded in a geodetic graph.
There are two challenges to overcome for such a result: first, to find a theoretical statement in the spirit of \cref{lem:stemple6} (necessarily with several more cases~-- see \cite{cu2021length} for some initial intents), and, second, to find a way to implement those cases efficiently.

\vspace{-2mm}
\paragraph*{Regular Geodetic Graphs.}

Very few regular geodetic graphs are known.
We observe that the cubic graph in~\cite{bosak1978geodetic} is actually part of an infinite family of regular geodetic graphs.
It appears as a special case of a construction presented in that paper; however, there are infinitely many such special cases (for details, see appendix).

\begin{restatable}{observation}{observationregular}
	\label{prop:observationregular}
		For each prime power $k$ there is a $(k+1)$-regular geodetic graph with $k^3 + 3k^2 + 3k + 2$ vertices.
\end{restatable}

Clearly, this family does not comprise all regular graphs as there are infinitely many 2-regular geodetic graphs (all odd cycles).
 Nevertheless, for larger $k$ it is not even known whether there are finitely or infinitely many $k$-regular graphs.

\clearpage
\bibliographystyle{splncs04}
\enlargethispage{2\baselineskip}
\bibliography{main}

\pagebreak
\appendix

\section{Appendix}

\subsection{Geodetic Subdivision}

As it did not fit in the main part, we give the proof of \Cref{prop:subdivcut} here in the appendix.

\propsubdivcut*

\begin{proof}
  We distinguish the following four cases:

  For $u, v \in S$ (resp. $T$) the shortest path in $G'$ is the same as in $G$, thus, it is unique.

  For $u \in S$, $v \in T$ all $u,v$ paths in $G'$ are longer than the corresponding paths in $G$ by at least $k$.
  Since the shortest path uses only one edge from the cut set, its length is increased by exactly $k$.
  Thus, the corresponding path in $G'$ remains unique.

  For $u \in S$ (resp. $T$) and $v$ one of the vertices subdividing the edge $\{x,y\}$  with $x \in S$ the shortest path is from $v$ to $x$ and then to $u$.
  The path $v$ to $y$ and then to $u$ using a different subdivided edge from the cut set has length at least
  \[d_G(u,y) + k + 1\geq d_G(u,x) + k + 1 > d_{G'}(u,v).\]
  The first inequality, $d_G(u,y) \geq d_G(u,x)$, is a consequence of $S$ being geodetically closed.
  We conclude that the shortest path in $G'$ is unique.

  If $u$ and $v$ are vertices subdividing two different edges, then the shortest path must be on the cycle of the subdivided edges and the shortest paths between their respective endpoints.
  By Condition~\ref{subdivcutCondII}, the cycle has odd length; thus, the shortest path is unique.\qed
\end{proof}

\subsection{Cliques and Independent Sets}\label{app:proof_clique_independentset_bicon}

Here we give the proof of \cref{lem:clique_independentset_bicon}.

\cliqueindepset*

\begin{proof}
	Assume that $G$ is not complete and let $H \subseteq V$ be a maximal clique containing $C$.
	From here on, our proof has two major steps.
	First we show that there is a vertex $v_* \in V$ such that there are at least two vertices $x\neq y \in H$ with $d(v_*, x) = d(v_*, y) =d(v_*, H)$.
	The second step is to show that $v_*$ has distance $d(v_*, H)$ to all vertices in $H$ and then conclude the statement of the proposition.

	Let $v_1$ be a vertex that is not in $H$.
	We know such a vertex exists, because $G$ is not complete.
	\Wlog $v_1$ has a neighbour in $H$.
	We denote by $u_1$ a vertex in $H$ that is adjacent to $v_1$.
	Removing $u_1$, there must still be a path from $v_1$ to any of the remaining vertices in $H$, because $G$ is biconnected.
	We pick one such path call the vertices on that path $v_1, \dots, v_k, u_k$, where $u_k \in H$.
	\Wlog $v_i \notin H$ for $1 \le i \le k$, otherwise we consider the shorter path.

	Now we make some definitions.
	Let $d_i = d(v_i, H) = \min\{d(v_i, h) \mid h \in H\}$ be the distance of $v_i$ to the clique.
	Let $S_i = \{h \in H \mid d(v_i, h) = d_i\}$ be the set of witnesses.
	If there is an $i \in \{1, \dots, k\}$ such that $|S_i| \ge 2$, then we have found our vertex $v_* = v_i$, that has distance $d(v_*, H)$ to two vertices in $H$, and we can continue with the next step of the proof.
	Now assume that $|S_i| = 1$ for all $1 \le i \le k$.
	We will show that this leads to a contradiction.
	Observe that $S_1 = \{u_1\}$ and $S_k = \{ u_k \}$, as $v_1$ and $v_k$ both are at distance one to the clique $H$.
	Thus, $S_1 \neq S_k$ and there is an index $i$ such that $S_i \neq S_{i+1}$.
	We write $S_i = \{u_i\}$ and $S_{i+1} = \{u_{i+1}\}$ and distinguish three cases:
	\begin{enumerate}
		\item $d_i = d_{i+1}$.
		Now there are two paths of length $d_i + 1$ from $u_i$ to $v_{i+1}$:
		One using the edge $\{u_i, u_{i+1}\}$ in the clique, and then the path from $u_{i+1}$ to $v_{i+1}$.
		The other using the path from $u_i$ to $v_i$ and then the edge $\{v_i, v_{i+1}\}$.
		Since shortest paths in geodetic graphs are unique, there must be a path from $u_i$ to $v_{i+1}$ of length at most $d_i$.
		In fact, that path must have length $d_i$, since $d_i = d_{i+1}$ is the minimum distance of $v_{i+1}$ to any vertex in $H$.
		Thus, $u_i \in S_{i+1}$, a contradiction.
		\item $d_i = d_{i+1} + 1$.
		There is a path of length $d_i$ from $u_{i+1}$ to $v_i$:
		Take the path from $u_{i+1}$ to $v_{i+1}$ of length $d_i - 1$ and then use the edge $\{v_{i+1}, v_i\}$.
		Thus, $u_{i+1} \in S_{i}$, a contradiction.
		\item $d_i + 1 = d_{i+1}$.
		Symmetric to the previous case.
	\end{enumerate}

	\noindent
	Now for the second step of the proof, we have a vertex $v_* \notin H$ with distance $d = d(v_*, H)$ to the clique $H$.
	This is witnessed by $S = \{u \in H \mid d(v_*, u) = d\}$ with $|S| \ge 2$.
	If $d = 1$, then by \Cref{lem:adjacent_part_of_clique} the set $H \cup \{ v_* \}$ is a clique, contradicting the maximality of $H$.
	Thus, $d \ge 2$.
	Let $u_1, u_2 \in S$ and consider a vertex $z \in H \setminus S$.
	There are two paths of length $d + 1$ from $z$ to $v_*$:
	For $i \in \{1, 2\}$ first use the edge $\{z, u_i\}$ and then the path from $u_i$ to $v_*$ of length $d$.
	Thus, there must be a shorter path of length at most $d$.
	In fact, the length must be precisely $d$, because $d$ is the minimum distance of $v_*$ to any vertex in $H$.
	Hence, $S = H$.
	Taking two vertices in $H$, by \Cref{lem:adjacent_part_of_clique} their neighbourhoods only overlap in the other clique vertices.
	Now, for $u \in H$ let $x_u$ denote the first vertex after $u$ on the unique shortest path from $u$ to $v_*$ and define $I' = \{x_u \mid u\in H \}$.
	We obtain the final independent set $I =\{x_u \mid u\in C \} \sse I'$.\qed
\end{proof}

\subsection{Enumerating Geodetic Graphs using \texttt{nauty}.}
In \Cref{lst:check_geodetic} one can see the prune function we used with \texttt{nauty}.
The \texttt{PRUNE} function is not only called for the final graph, but also for intermediate graphs:
Graphs are constructed by adding one vertex after another.
Each time a vertex is added, the prune function is called with the intermediate graph, which is an induced subgraph of the final graph.
We take advantage of this by checking whether each 4-cycle in an intermediate graph induces a complete subgraph as required by \cref{lem:diamonds}.

The \texttt{nauty} library represents graphs as adjacency matrices.
\texttt{MAXN} is the maximum number of nodes of the graph supported by \texttt{nauty}.
Each line of the adjacency matrix is stored as a bit vector (in our case in a single 32-bit integer).
The type aliases \texttt{graph} and \texttt{set} refer to the same integer type.
Our code assumes that a line of the adjacency matrix is exactly one \texttt{set}, thus, \texttt{MAXM == 1}.
A graph is represented using an array of \texttt{set}.
The constant \texttt{bit[i]} is a \texttt{set} with the $i$-th most significant bit set.
For more details we refer the reader to the documentation of \texttt{nauty}.

\begin{lstfloat}[H]
  \begin{lstlisting}[language=C]
#include "gtools.h"

_Static_assert(MAXM == 1, "MAXM == 1");
_Static_assert(MAXN >= 1, "MAXN >= 1");

int check_geodetic(graph *g, int n, int maxn) {
  int rounds = n - 1;
  if (n != maxn) {
    /* Intermediate graph: one iteration. */
    rounds = 1;
  }
  /* vertices reachable in round steps */
  graph p[MAXN * MAXM];
  for (int k = 0; k <= n * MAXM; k++) {
    p[k] = g[k];
  }
  for (int r = 1; r <= rounds; r++) {
    for (int i = 0; i < n; i++) {
      /* compute vertices at distance r+1 from i */
      set reach = p[i] & ~bit[i];
      set not_reach = ~(p[i] | bit[i]);
      for (int j = 0; j < n; j++) {
        if (reach & bit[j]) {
          // reach via j
          set reach_new = g[j] & not_reach;
          if (p[i] & reach_new) {
            return 1;
          }
          p[i] |= reach_new;
        }
      }
    }
  }
  return 0;
}
  \end{lstlisting}
  \caption{Function to check whether a graph is geodetic, using \texttt{nauty}.}
  \label{lst:check_geodetic}
\end{lstfloat}

Let $A$ be the adjacency matrix of the graph.
The idea of the algorithm is to compute powers $A^r$ of the adjacency matrix, for $1 \le r \le n$.
However, we only store whether or not an entry is 0.
Thus, we can use the same type of bitvector, that is used to store the graph.
If an entry $(i,j)$ is $0$ in some $A^r$, but not zero in $A^{r+1}$, and there is more than one path of length $r+1$, then the algorithm will return $1$ in line 27, indicating that the graph is not geodetic.
For intermediate graphs, only $A^2$ is computed, that is we only check for 4-cycles.

We have gathered some statistics in \Cref{tab:geng_graphs}: in the second and third column we show respectively the number of intermediate and final graphs the prune function has been called with.
In the last column we show the number of biconnected geodetic graphs.

\begin{table}[t]
  \centering
  \small
  \caption{Statistics on the \texttt{nauty} based Algorithm.}
  \label{tab:geng_graphs}
  \begin{tabular}{rrrr}
    \toprule
    \textbf{Vertices} & \textbf{Intermediate Graphs} & \textbf{Final Graphs} & \textbf{Geo.} \\
    \midrule
    3                 & 2                            & 1                     & 1             \\
    4                 & 5                            & 3                     & 1             \\
    5                 & 13                           & 7                     & 2             \\
    6                 & 105                          & 30                    & 1             \\
    7                 & 290                          & 132                   & 3             \\
    8                 & 984                          & 844                   & 1             \\
    9                 & 4346                         & 6652                  & 3             \\
    10                & 26768                        & 64627                 & 4             \\
    11                & 216599                       & 747768                & 3             \\
    12                & 2156780                      & 10226920              & 1             \\
    13                & 26681137                     & 163847610             & 9             \\
    14                & 395136409                    & 3057525463            & 2             \\
    15                & 6869665268                   & 66140495796           & 4             \\
    16                & 139198227567                 & 1652830274122         & 8             \\
    \bottomrule
  \end{tabular}
\end{table}

\subsection{Details on the Custom Search Algorithm}\label{app:pseudocode-example}

In the following we present pseudocode for the two generation steps of the search algorithm.
The generation of sibling-edges is realized by the function \textsc{GenCliques}($T$, $E$, $v$), where $T$ is the tree, $E$ the edge set of the graph being generated, and $v$ is the vertex whose children are to be partitioned into cliques.
$\operatorname{children}(v)$ refers to the children of $v$ in the tree $T$,
$\operatorname{NextOnLevel}(v)$ is the next vertex on the same level as $v$, and
$\operatorname{NextLevel}(v)$ is the set of vertices on the level below $v$.
The algorithm is initiated by calling \textsc{GenCliques}($T$, $E$, $r$) where $r$ is the root of $T$ and $E$ the set of edges of $T$.
We give the pseudo-code of the function:

\begin{algorithmic}[1]\small
\Function{GenCliques}{$T$, $E$, $v$}
  \For{each partition $P$ of $\operatorname{children}(v)$}
    \State $E' \gets E \cup {} \bigcup\bigl\{\binom{C}{2} \mid C \in P\bigr\} $
    \If {exists $\operatorname{NextOnLevel}(v)$}
      \State \textsc{GenCliques}($T$, $E'$, $\operatorname{NextOnLevel}(v)$)
    \ElsIf{$\operatorname{NextLevel}(v) \neq \emptyset$}
    \State \textsc{GenNeighbour}($T$, $E'$, $\operatorname{First}(\operatorname{NextLevel}(v))$)
    \ElsIf{$E'$ is geodetic and biconnected}
    \State Print $E'$
    \EndIf
  \EndFor
\EndFunction
\end{algorithmic}

In the third line the edge set is extended by adding cliques according to the partition of the children of $v$.
If $v$ is not the last vertex on the level, then in line 5, the function recursively calls itself for the next vertex.
If $v$ is the last vertex on the level, but not the last vertex in the tree, then the \textsc{GenNeighbour} is called, which realizes the generation of non-sibling edges.
If $v$ is the last vertex in the tree, then the recursion is terminated and we check whether the generated tree is geodetic and biconnected.
We give the pseudo-code of the \textsc{GenNeighbour} function:

\begin{algorithmic}[1]\small
\Function{GenNeighbour}{$T$, $E$, $v$}
  \State $K \gets \{u \in V \mid \operatorname{level}(u) = \operatorname{level}(v) \land u > v \land \operatorname{par}(u) \neq \operatorname{par}(v) \}$
  \For{each $U \subseteq K$}
    \State $E' \gets E \cup \bigl\{\{v,u\} \mid  u \in U \bigr\}$
    \State \textsc{Prune}($T$, $E'$, $v$)
    \If {exists $\operatorname{NextOnLevel}(v)$}
      \State \textsc{GenNeighbour}($T$, $E'$, $\operatorname{NextOnLevel}(v)$)
    \Else
      \State \textsc{GenCliques}($T$, $E'$, $\operatorname{First}(\operatorname{Level}(v))$)
    \EndIf
  \EndFor
\EndFunction
\end{algorithmic}

As the vertices are traversed in level-order, a call to \textsc{GenNeighbour}($T$, $E$, $v$) only generates edges to vertices that have not been traversed yet.
Recall that we assume an order on the vertices on the same level.
$ u > v$ refers to that order.
The parent of $u$ is denoted by $\operatorname{par}(u)$.
In line 2, the set $K$ of these vertices is computed.
The \textsc{Prune} function checks whether the subset $U$ of $K$ still can lead to the discovery of new geodetic graphs and aborts if that is not the case.
\Cref{fig:example-search-space} illustrates running the algorithm on a single tree.

\subsubsection{Pruning Small Cycles}
\label{app:small-cycles}

\paragraph*{4-cycles.} By \cref{lem:diamonds}, every 4-cycle needs to be a clique.
From that we derive the following restrictions on non shortest-path tree edges.
\begin{enumerate}
  \item Let $u, v_1, v_2$ be vertices on the same level such that $v_1$ and $v_2$ have the same parent but $u$ has a different parent.
  Then it is not allowed to have both of the edges $\{u, v_1\}$ and $\{u, v_2\}$.
  One of them is allowed, though.
  Depending on the order in which we visit the vertices $u,v_1,v_2$, we either implement this as part of the pruning or via the forbidden edge set.
  \item Let $u, v, w, x$ be vertices on the same level.
  If $(u, v, w, x)$ is a 4-cycle, then the four vertices need to be a clique.
\end{enumerate}

\paragraph*{6-cycles.}
We also make use of the following fact due to Stemple.
\begin{lemma}
[Theorem~3.4~of~\cite{stemple1974geodetic}]
  \label{lem:stemple6}
  If a geodetic graph $G$ contains a cycle of length six, then one of the following three cases applies:
  \begin{enumerate}
    \item The vertices of the 6-cycle induce a complete subgraph.
    \item The vertices can be labelled cyclically $1, 2, \dots, 6$ such that there is an edge $\{1, 3\}$ and a path from $2$ to $5$  of length two whose intermediate vertex is not on the 6-cycle.
    \item Labelling as above, there are three paths of length two, connecting $1$ to $4$, $2$ to $5$ and $3$ to $6$, 
    whose intermediate vertices are all different and not on the 6-cycle.
  \end{enumerate}
\end{lemma}

\begin{figure}[H]
  \centering
  \input{graphs/example-search-space}
  \caption{Search space of the backtracking search for a single tree.
  We found one geodetic graph, a subdivision of $K_4$.
  For each unfinished graph without children, the pruning heuristic due to which it is discarded is displayed under the graph.}
  \label{fig:example-search-space}
\end{figure}

We check whether the conditions of \Cref{lem:stemple6} are fulfilled for the 6-cycles shown in \Cref{fig:prune-stemple}.
Vertices drawn above others are on a higher level in the tree.
We decided on those patterns because we can detect them efficiently and we can check the conditions of the lemma efficiently.
In particular, we can easily look up parent vertices in the tree, and, for two vertices,  thanks to the adjacency matrix representation of the graphs, we can enumerate all distance two paths between them with a simple binary \textsc{And}.
All of the patterns we consider have at most four vertices on the same level.
That rules out the first case of \Cref{lem:stemple6}.

In a 6-cycle, a shortcut edge is an edge between two vertices which are at distance two in the 6-cycle.
The conditions on non shortest-path tree edges and on 4-cycles derived from \Cref{lem:unique_predessesor} and \cref{lem:diamonds} allow for at most one shortcut edge per 6-cycle, provided that is is of one of the patterns that we consider.
Instead of explicitly checking which  shortcut edges are present, we use the following observation.
In cases 2 and 3 of \Cref{lem:stemple6}, there is a unique path of length two between any pair of opposing vertices on the cycle.
So for each 6-cycle corresponding to one of the patterns in \Cref{fig:prune-stemple}, we test whether those unique paths exist.

\begin{figure}[t]
  \centering
  \begin{subfigure}[b]{0.19\textwidth}
    \centering
    \begin{tikzpicture}[style=thick,xscale=0.8,yscale=0.48]
  \coordinate (gv) at (1, 2);
  \coordinate (px) at (0.3, 1);
  \coordinate (pv) at (1.7, 1);
  \coordinate (x) at (0, 0);
  \coordinate (y) at (1, 0);
  \coordinate (v) at (2, 0);
  \fill (gv) circle[shift only,radius=2pt] node[above] {$gp_v$};
  \fill (px) circle[shift only,radius=2pt] node[above left] {$p_x$};
  \fill (pv) circle[shift only,radius=2pt] node[above right] {$p_v$};
  \fill (x) circle[shift only,radius=2pt] node[below] {$x$};
  \fill (y) circle[shift only,radius=2pt] node[below] {$y$};
  \fill (v) circle[shift only,radius=2pt] node[below] {$v$};
  \draw (x) -- (y) -- (v) -- (pv) -- (gv) -- (px) -- cycle;
\end{tikzpicture}
  \end{subfigure}
  \hspace{-1cm}
  \hfill
  \begin{subfigure}[b]{0.19\textwidth}
    \centering
    \begin{tikzpicture}[style=thick,xscale=0.8,yscale=0.48]
  \coordinate (gv) at (1.35, 2);
  \coordinate (pu) at (0.3, 1);
  \coordinate (c) at (1, 1);
  \coordinate (pv) at (1.7, 1);
  \coordinate (u) at (0, 0);
  \coordinate (v) at (2, 0);
  \fill (gv) circle[shift only,radius=2pt] node[above] {$gp_v$};
  \fill (pu) circle[shift only,radius=2pt] node[above left] {$p_u$};
  \fill (pv) circle[shift only,radius=2pt] node[above right] {$p_v$};
  \fill (u) circle[shift only,radius=2pt] node[below] {$u$};
  \fill (c) circle[shift only,radius=2pt] node[below] {$c$};
  \fill (v) circle[shift only,radius=2pt] node[below] {$v$};
  \draw (pu) -- (u) -- (v) -- (pv) -- (gv) -- (c) -- cycle;
\end{tikzpicture}
  \end{subfigure}
  \hspace{-1cm}
  \hfill
  \begin{subfigure}[b]{0.18\textwidth}
    \centering
    \begin{tikzpicture}[style=thick,xscale=-0.8,yscale=0.48]
  \coordinate (gv) at (1.35, 2);
  \coordinate (pu) at (0.3, 1);
  \coordinate (c) at (1, 1);
  \coordinate (pv) at (1.7, 1);
  \coordinate (u) at (0, 0);
  \coordinate (v) at (2, 0);
  \fill (gv) circle[shift only,radius=2pt] node[above] {$gp_v$};
  \fill (pu) circle[shift only,radius=2pt] node[above right] {$p_w$};
  \fill (pv) circle[shift only,radius=2pt] node[above left] {$p_v$};
  \fill (x) circle[shift only,radius=2pt] node[below] {$w$};
  \fill (c) circle[shift only,radius=2pt] node[below] {$c$};
  \fill (v) circle[shift only,radius=2pt] node[below] {$v$};
  \draw (pu) -- (u) -- (v) -- (pv) -- (gv) -- (c) -- cycle;
\end{tikzpicture}
  \end{subfigure}
  \hspace{-1cm}
  \hfill
  \begin{subfigure}[b]{0.20\textwidth}
    \centering
    \begin{tikzpicture}[style=thick,xscale=0.8,yscale=0.48]
  \coordinate (c) at (1, 1);
  \coordinate (px) at (0.3, 1);
  \coordinate (pv) at (1.7, 1);
  \coordinate (x) at (0, 0);
  \coordinate (y) at (1, 0);
  \coordinate (v) at (2, 0);
  \fill (c) circle[shift only,radius=2pt] node[above] {$c$};
  \fill (px) circle[shift only,radius=2pt] node[above] {$p_x$};
  \fill (pv) circle[shift only,radius=2pt] node[above] {$p_v$};
  \fill (x) circle[shift only,radius=2pt] node[below] {$x$};
  \fill (y) circle[shift only,radius=2pt] node[below] {$y$};
  \fill (v) circle[shift only,radius=2pt] node[below] {$v$};
  \draw (x) -- (y) -- (v) -- (pv) -- (c) -- (px) -- cycle;
\end{tikzpicture}
  \end{subfigure}
  \hspace{-1cm}
  \hfill
  \begin{subfigure}[b]{0.24\textwidth}
    \centering
    \begin{tikzpicture}[style=thick,xscale=0.8,yscale=0.48]
  \coordinate (pu) at (-0.5, 1);
  \coordinate (pv) at (1.5, 1);
  \coordinate (u) at (-1, 0);
  \coordinate (x) at (0, 0);
  \coordinate (y) at (1, 0);
  \coordinate (v) at (2, 0);
  \fill (u) circle[shift only,radius=2pt] node[below] {$u$};
  \fill (pu) circle[shift only,radius=2pt] node[above] {$p_u$};
  \fill (pv) circle[shift only,radius=2pt] node[above] {$p_v$};
  \fill (x) circle[shift only,radius=2pt] node[above right] {$x$};
  \fill (y) circle[shift only,radius=2pt] node[above] {$y$};
  \fill (v) circle[shift only,radius=2pt] node[below] {$v$};
  \draw (x) -- (y) -- (pv) -- (v) edge[bend left] (u);
  \draw (u) -- (pu) -- (x);
\end{tikzpicture}
  \end{subfigure}
  \caption{6-cycles considered for pruning.}
  \label{fig:prune-stemple}
\end{figure}

\subsection{Running Times}\label{app:running_times}

In \Cref{tab:running_times} we compare the running times of the \texttt{nauty}-based algorithm to our custom search algorithm, as described in \Cref{subsec:custom-search-algorithm}.
We ran both algorithms on a machine with a Ryzen 9 5900X (12 cores/ 24 threads) processor and 128\,GB of RAM.

Additionally, we collected running times with some of the pruning techniques and optimizations disabled, to demonstrate their effectiveness.
In particular, we disabled to pruning of 6-cycles (Stemple Pruning), the check whether it is still possible to obtain a biconnected graph from an intermediate graph (Biconnectivity), and the optimization enforcing that each clique vertex has a neighbour outside of the clique (Clique Neighbour).

\begin{table}[t]
  \centering
  \footnotesize
  \setlength{\tabcolsep}{2pt}
  \caption{Comparison of Running Times.}
  \label{tab:running_times}
  \begin{tabular}{lrrrrrrrrrrrrrrr}
    \toprule
    \footnotesize \textbf{Vertices} & 13 & 14 & 15 & 16 & 17 & 18 & 19 & 20 & 21 & 22 & 23 & 24 & 25 \\
    \midrule
    \footnotesize \texttt{nauty}-based  & 8.9s & 3.4m & 93m & 56h &            \\
    \footnotesize algorithm \\
    \footnotesize Custom search & 10ms & 29ms & 87ms & 0.5s & 1.0s & 2.3s & 11s & 58s & 5m & 41m & 4.8h & 34h & 350h  \\
    \footnotesize algorithm \\
    \footnotesize Custom without & 90ms & 0.45s & 8.1s & 94s & 14m & 50m & 5.9h & \\
    \footnotesize Stemple Pruning \\
    \footnotesize Custom without & 7ms & 24ms & 0.1s & 0.46s & 1.6s & 2.8s & 11s & 63s & 6.7m & 47m & 6h \\
    \footnotesize Biconnectivity \\
    \footnotesize Custom without & 14ms & 51ms & 0.2s & 1.1s & 2.9s & 5.7s & 26s & 115s & 13m & 82m & 9.7h \\
    \footnotesize Clique Neighbour \\
    \bottomrule
  \end{tabular}
\end{table}

\subsection{New Constructions}
In this section we give formal proofs that our new constructions are indeed geodetic graphs.

\begin{proposition}
  The graph $H(m, n, p, s)$ is geodetic for all $m, n, p \ge 2$ and $s \ge 0$.
\end{proposition}

\begin{proof}
  Take two vertices in the subgraph $h(m, n, s)$ and a path joining them.
  If the path contains vertices not part of the subgraph $h(m, n, s)$, then it is not a shortest path:
  If a path segment leaves and enters the subgraph though the same clique, then the segment can be replaced with a clique edge.
  If it leaves through the $K_m$ and enters through the $K_n$ (or vice versa), then the segment has length at least $4s+6$ and can be replaced by the shortest path inside the subgraph $h(m, n, s)$, which has length $2s+3$.
  Since $h(m, n, s)$ is geodetic, shortest paths between two vertices in the subgraph are unique.
  By symmetry this also applies to the subgraphs $h(n, p, s)$ and $h(p, m, s)$.

  Now consider two vertices $u$ and $v$ not in the same subgraph and assume that there are two distinct paths of length $d$ from $u$ to $v$.
  We have to show that then the shortest path has length less than $d$:
  The two paths form a cycle of even length.
  The smallest even cycle not contained in one of the subgraphs has diameter $3s + 5$ (\ie length $6s + 10$).
  This is because it either must traverse one of the subgraphs entirely from left to right and back and one of the subgraphs until the middle or it must go around the complete ``cycle'' once (see \cref{fig:family-H}).
  Thus, $d \ge 3s + 5$.
  However, the diameter of $H(m, n, p, s)$ is $3s + 4$; thus, the shortest path from $u$ to $v$ has length at most $3s + 4$ and is unique.\qed
\end{proof}

\begin{proposition}
	The graph $F_k$ is geodetic for any odd integer $k \ge 3$.
\end{proposition}

\begin{proof}
	The graphs $F_3$ and $F_5$ can be easily checked by hand or computer. To prove that $F_k$ for $k \geq 7$ is geodetic, we apply \cref{prop:subdivcut} to the graph $F_5$:

	Let $S = \{ b, u_2, u_3, u_4, v_2, v_3, v_4\}$ and $T$ the remaining vertices (\ie in the picture in \cref{fig:Fk}, we cut $F_5$ via a horizontal line).
  Then $S$ and $T$ are geodetically closed, the cut edges are the four edges $\{u_1,u_2\},\{u_4,u_5\},\{v_1,v_2\},\{v_4,v_5\}$ and, condition~2 in \cref{prop:subdivcut} is satisfied.
  The latter can be seen easily in~\cref{fig:Fk} (by symmetry, we only need to consider the two cases that either both edges are in the same of the two 5-cycles or that they are in different 5-cycles).
  Hence, by \cref{prop:subdivcut} it remains to observe that by putting $\ell$ new vertices on every cut edge, we obtain the graph $F_{5 + 2\ell}$.\qed
\end{proof}

\subsection{Star Operation}

\Cref{fig:k6-star} shows the result of applying the star operation as defined in~\cite{parthasarathy1982some} to a $K_6$.
To obtain the graph on the left with 24 vertices, start with an $\mathcal{O}$-cover of $K_6$ with 4 triangles and 4 single edges.
The graph on the right with 25 vertices is obtained starting with an $\mathcal{O}$-cover of $K_6$ with one $K_4$, one triangle and 6 single edges.

\begin{figure}[t]
  \centering
  \begin{subfigure}[b]{0.48\linewidth}
    \centering
    \begin{tikzpicture}[style=thick,scale=0.8]

  \foreach \x in {90,210,330}{

    \coordinate (a) at ($(\x:2cm)+(\x-90-90:0.4cm)$);
    \coordinate (b) at ($(\x:2cm)+(\x-90-210:0.4cm)$);
    \coordinate (c) at ($(\x:2cm)+(\x-90-330:0.4cm)$);
    \coordinate (b2) at ($(\x+120:2cm)+(\x+120-90-210:0.4cm)$);
    \coordinate (c3) at ($(\x+240:2cm)+(\x+240-90-330:0.4cm)$);
    \coordinate (x) at ($(\x:0.5cm)!0.5!(a)$);
    \coordinate (b2c3) at ($(b2)!0.5!(c3)$);
    \coordinate (d1) at ($(x)!0.3!(b2c3) + (\x-90:0.5cm)$);
    \coordinate (d2) at ($(x)!0.75!(b2c3) + (\x-90:0.5cm)$);

    \fill (\x:0.4cm) circle[shift only,radius=2pt];
    \fill (a) circle[shift only,radius=2pt];
    \fill (b) circle[shift only,radius=2pt];
    \fill (c) circle[shift only,radius=2pt];

    \draw (\x:0.4cm) -- (\x+120:0.4cm);
    \draw (a) -- (b) -- (c) -- cycle;

    \draw (\x:0.5cm) -- (a);
    \fill (x) circle[shift only,radius=2pt];
    \fill (b2c3) circle[shift only,radius=2pt];
    \fill (d1) circle[shift only,radius=2pt];
    \fill (d2) circle[shift only,radius=2pt];

    \draw (b2) -- (c3);
    \draw (x) -- (d1) -- (d2) -- (b2c3);
  }
\end{tikzpicture}
  \end{subfigure}
  \hfill
  \begin{subfigure}[b]{0.48\linewidth}
    \centering
    \begin{tikzpicture}[style=thick,scale=0.85]
        \foreach \x in {90,210,330}{
          \draw (\x:2.4cm) -- (\x+120:2.4cm);
          \fill (\x:2.4cm) circle[shift only,radius=2pt];
        }
        \coordinate (l) at (210:1.6cm);
        \coordinate (r) at (330:1.6cm);
        \coordinate (t) at (90:1.6cm);
        \coordinate (0) at (-0.6, 0.8);
        \coordinate (1) at (0, 0.8);
        \coordinate (2) at (0.6, 0.8);
        \coordinate (0a) at (-0.8, 0.5);
        \coordinate (1a) at (0, 0.5);
        \coordinate (2a) at (0.8, 0.5);
        \draw (90:2.4cm) -- (t);
        \draw (210:2.4cm) -- (l);
        \draw (330:2.4cm) -- (r);
        \draw (t) -- (0);
        \draw (t) -- (1);
        \draw (t) -- (2);
        \draw (0) -- (1) -- (2);
        \draw (0) edge[bend left] (2);
        \draw (l) -- (0a);
        \draw (l) -- (1a);
        \draw (l) -- (2a);
        \draw (r) -- (0a);
        \draw (r) -- (1a);
        \draw (r) -- (2a);
        \draw (0) -- (0a);
        \draw (1) -- (1a);
        \draw (2) -- (2a);
        \fill (0) circle[shift only,radius=2pt];
        \fill (1) circle[shift only,radius=2pt];
        \fill (2) circle[shift only,radius=2pt];
        \fill (0a) circle[shift only,radius=2pt];
        \fill (1a) circle[shift only,radius=2pt];
        \fill (2a) circle[shift only,radius=2pt];
        \fill (l) circle[shift only,radius=2pt];
        \fill (r) circle[shift only,radius=2pt];
        \fill (t) circle[shift only,radius=2pt];
        \fill (90:2cm) circle[shift only,radius=2pt];
        \fill ($(l)!0.333!(0a)$) circle[shift only,radius=2pt];
        \fill ($(l)!0.666!(0a)$) circle[shift only,radius=2pt];
        \fill ($(l)!0.333!(1a)$) circle[shift only,radius=2pt];
        \fill ($(l)!0.666!(1a)$) circle[shift only,radius=2pt];
        \fill ($(l)!0.333!(2a)$) circle[shift only,radius=2pt];
        \fill ($(l)!0.70!(2a)$) circle[shift only,radius=2pt];
        \fill ($(r)!0.333!(2a)$) circle[shift only,radius=2pt];
        \fill ($(r)!0.666!(2a)$) circle[shift only,radius=2pt];
        \fill ($(r)!0.333!(1a)$) circle[shift only,radius=2pt];
        \fill ($(r)!0.666!(1a)$) circle[shift only,radius=2pt];
        \fill ($(r)!0.333!(0a)$) circle[shift only,radius=2pt];
        \fill ($(r)!0.70!(0a)$) circle[shift only,radius=2pt];
      \end{tikzpicture}
  \end{subfigure}
  \caption{The star operation applied to the $K_6$ in two different ways.}
  \label{fig:k6-star}
\end{figure}

\subsection{Pulling of Geodetic Subgraphs}
\label{sec:pull}

Plesnik introduced the operation of pulling a geodetic subgraph homeomorphic to a complete graph as a method for constructing a geodetic graph from another~\cite{plesnik1984construction}.
Applying this operation twice to a geodetic subdivision of $K_6$ obtained from the labelling $(1, 1, 1, 1, 1, 0)$, we obtain a geodetic graph with $25$ vertices.
This is visualized in \Cref{fig:k6-pull}.

\begin{figure}[t]
	\centering
	\begin{subfigure}[b]{0.3\linewidth}
		\centering
		\begin{tikzpicture}[style=thick,xscale=0.85, yscale=0.6]

        \coordinate (1) at (0, 0);
        \coordinate (2) at (2, 0);
        \coordinate (3) at (1.5, 2);
        \coordinate (4) at (0, 4);
        \coordinate (5) at (2, 4);
        \coordinate (6) at (3, 2);
        \coordinate (61) at (3, 0.5);
        \coordinate (62) at (2.5, 1);
        \coordinate (63) at (2.5, 2);
        \coordinate (64) at (2.5, 3);
        \coordinate (65) at (3, 3.5);
        \draw (1) -- (2);
        \draw (1) -- (3);
        \draw (1) -- (4);
        \draw (1) -- (5);
        \draw (1) to [bend right=40] (61);
        \draw (61) -- (6);
        \draw (2) -- (3);
        \draw (2) -- (4);
        \draw (2) -- (5);
        \draw (2) -- (6);
        \draw (3) -- (4);
        \draw (3) -- (5);
        \draw (3) -- (6);
        \draw (4) -- (5);
        \draw (5) -- (6);
        \draw (4) to [bend left=40] (65);
        \draw (65) -- (6);
        \fill (1) circle[shift only,radius=2pt];
        \fill (2) circle[shift only,radius=2pt];
        \fill (3) circle[shift only,radius=2pt];
        \fill (4) circle[shift only,radius=2pt];
        \fill (5) circle[shift only,radius=2pt];
        \fill (6) circle[shift only,radius=2pt];
        \fill (61) circle[shift only,radius=2pt];
        \fill (62) circle[shift only,radius=2pt];
        \fill (63) circle[shift only,radius=2pt];
        \fill (64) circle[shift only,radius=2pt];
        \fill (65) circle[shift only,radius=2pt];
        \fill ($(1)!0.333!(2)$) circle[shift only,radius=2pt];
        \fill ($(1)!0.666!(2)$) circle[shift only,radius=2pt];
        \fill ($(1)!0.333!(3)$) circle[shift only,radius=2pt];
        \fill ($(1)!0.666!(3)$) circle[shift only,radius=2pt];
        \fill ($(1)!0.333!(4)$) circle[shift only,radius=2pt];
        \fill ($(1)!0.666!(4)$) circle[shift only,radius=2pt];
        \fill ($(1)!0.333!(5)$) circle[shift only,radius=2pt];
        \fill ($(1)!0.666!(5)$) circle[shift only,radius=2pt];
        \fill ($(2)!0.333!(3)$) circle[shift only,radius=2pt];
        \fill ($(2)!0.666!(3)$) circle[shift only,radius=2pt];
        \fill ($(2)!0.333!(4)$) circle[shift only,radius=2pt];
        \fill ($(2)!0.666!(4)$) circle[shift only,radius=2pt];
        \fill ($(2)!0.333!(5)$) circle[shift only,radius=2pt];
        \fill ($(2)!0.666!(5)$) circle[shift only,radius=2pt];
        \fill ($(3)!0.333!(4)$) circle[shift only,radius=2pt];
        \fill ($(3)!0.666!(4)$) circle[shift only,radius=2pt];
        \fill ($(3)!0.333!(5)$) circle[shift only,radius=2pt];
        \fill ($(3)!0.666!(5)$) circle[shift only,radius=2pt];
        \fill ($(4)!0.333!(5)$) circle[shift only,radius=2pt];
        \fill ($(4)!0.666!(5)$) circle[shift only,radius=2pt];
\end{tikzpicture}
	\end{subfigure}
	\hfill
	\begin{subfigure}[b]{0.3\linewidth}
		\centering
		\begin{tikzpicture}[style=thick,xscale=0.85, yscale=0.6]

        \coordinate (1) at (0, 0);
        \coordinate (2) at (2, 0);
        \coordinate (3) at (1.5, 2);
        \coordinate (1a) at (0.6, 0.6);
        \coordinate (2a) at (1.4, 0.6);
        \coordinate (3a) at (1, 1.4);
        \coordinate (4) at (0, 4);
        \coordinate (5) at (2, 4);
        \coordinate (6) at (3, 2);
        \coordinate (61) at (3, 0.5);
        \coordinate (62) at (2.5, 1);
        \coordinate (63) at (2.5, 2);
        \coordinate (64) at (2.5, 3);
        \coordinate (65) at (3, 3.5);
        \draw (1a) -- (2a);
        \draw (1a) -- (3a);
        \draw (1) -- (4);
        \draw (1) -- (5);
        \draw (1) to [bend right=40] (61);
        \draw (61) -- (6);
        \draw (2a) -- (3a);
        \draw (2) -- (4);
        \draw (2) -- (5);
        \draw (2) -- (6);
        \draw (3) -- (4);
        \draw (3) -- (5);
        \draw (3) -- (6);
        \draw (4) -- (5);
        \draw (5) -- (6);
        \draw (4) to [bend left=40] (65);
        \draw (65) -- (6);
        \draw (1) -- (1a);
        \draw (2) -- (2a);
        \draw (3) -- (3a);
        \fill (1) circle[shift only,radius=2pt];
        \fill (2) circle[shift only,radius=2pt];
        \fill (3) circle[shift only,radius=2pt];
        \fill (1a) circle[shift only,radius=2pt];
        \fill (2a) circle[shift only,radius=2pt];
        \fill (3a) circle[shift only,radius=2pt];
        \fill (4) circle[shift only,radius=2pt];
        \fill (5) circle[shift only,radius=2pt];
        \fill (6) circle[shift only,radius=2pt];
        \fill (61) circle[shift only,radius=2pt];
        \fill (62) circle[shift only,radius=2pt];
        \fill (63) circle[shift only,radius=2pt];
        \fill (64) circle[shift only,radius=2pt];
        \fill (65) circle[shift only,radius=2pt];
        \fill ($(1)!0.333!(4)$) circle[shift only,radius=2pt];
        \fill ($(1)!0.666!(4)$) circle[shift only,radius=2pt];
        \fill ($(1)!0.333!(5)$) circle[shift only,radius=2pt];
        \fill ($(1)!0.666!(5)$) circle[shift only,radius=2pt];
        \fill ($(2)!0.333!(4)$) circle[shift only,radius=2pt];
        \fill ($(2)!0.666!(4)$) circle[shift only,radius=2pt];
        \fill ($(2)!0.333!(5)$) circle[shift only,radius=2pt];
        \fill ($(2)!0.666!(5)$) circle[shift only,radius=2pt];
        \fill ($(3)!0.333!(4)$) circle[shift only,radius=2pt];
        \fill ($(3)!0.666!(4)$) circle[shift only,radius=2pt];
        \fill ($(3)!0.333!(5)$) circle[shift only,radius=2pt];
        \fill ($(3)!0.666!(5)$) circle[shift only,radius=2pt];
        \fill ($(4)!0.333!(5)$) circle[shift only,radius=2pt];
        \fill ($(4)!0.666!(5)$) circle[shift only,radius=2pt];
\end{tikzpicture}
	\end{subfigure}
	\hfill
	\begin{subfigure}[b]{0.3\linewidth}
		\centering
		\begin{tikzpicture}[style=thick,xscale=0.85, yscale=0.6]

        \coordinate (1) at (0, 0);
        \coordinate (2) at (2, 0);
        \coordinate (3) at (1.5, 2);
        \coordinate (3b) at (1, 2.6);
        \coordinate (1a) at (0.6, 0.6);
        \coordinate (2a) at (1.4, 0.6);
        \coordinate (3a) at (1, 1.4);
        \coordinate (4) at (0, 4);
        \coordinate (4b) at (0.6, 3.4);
        \coordinate (5) at (2, 4);
        \coordinate (5b) at (1.4, 3.4);
        \coordinate (6) at (3, 2);
        \coordinate (61) at (3, 0.5);
        \coordinate (62) at (2.5, 1);
        \coordinate (63) at (2.5, 2);
        \coordinate (64) at (2.5, 3);
        \coordinate (65) at (3, 3.5);
        \draw (1a) -- (2a);
        \draw (1a) -- (3a);
        \draw (1) -- (4);
        \draw (1) -- (5);
        \draw (1) to [bend right=40] (61);
        \draw (61) -- (6);
        \draw (2a) -- (3a);
        \draw (2) -- (4);
        \draw (2) -- (5);
        \draw (2) -- (6);
        \draw (3b) -- (4b);
        \draw (3b) -- (5b);
        \draw (3) -- (6);
        \draw (4b) -- (5b);
        \draw (5) -- (6);
        \draw (4) to [bend left=40] (65);
        \draw (65) -- (6);
        \draw (1) -- (1a);
        \draw (2) -- (2a);
        \draw (3) -- (3a);
        \draw (3) -- (3b);
        \draw (4) -- (4b);
        \draw (5) -- (5b);
        \fill (1) circle[shift only,radius=2pt];
        \fill (2) circle[shift only,radius=2pt];
        \fill (3) circle[shift only,radius=2pt];
        \fill (1a) circle[shift only,radius=2pt];
        \fill (2a) circle[shift only,radius=2pt];
        \fill (3a) circle[shift only,radius=2pt];
        \fill (4) circle[shift only,radius=2pt];
        \fill (5) circle[shift only,radius=2pt];
        \fill (3b) circle[shift only,radius=2pt];
        \fill (4b) circle[shift only,radius=2pt];
        \fill (5b) circle[shift only,radius=2pt];
        \fill (6) circle[shift only,radius=2pt];
        \fill (61) circle[shift only,radius=2pt];
        \fill (62) circle[shift only,radius=2pt];
        \fill (63) circle[shift only,radius=2pt];
        \fill (64) circle[shift only,radius=2pt];
        \fill (65) circle[shift only,radius=2pt];
        \fill ($(1)!0.333!(4)$) circle[shift only,radius=2pt];
        \fill ($(1)!0.666!(4)$) circle[shift only,radius=2pt];
        \fill ($(1)!0.333!(5)$) circle[shift only,radius=2pt];
        \fill ($(1)!0.666!(5)$) circle[shift only,radius=2pt];
        \fill ($(2)!0.333!(4)$) circle[shift only,radius=2pt];
        \fill ($(2)!0.666!(4)$) circle[shift only,radius=2pt];
        \fill ($(2)!0.333!(5)$) circle[shift only,radius=2pt];
        \fill ($(2)!0.666!(5)$) circle[shift only,radius=2pt];
\end{tikzpicture}
	\end{subfigure}
	\caption{The pulling operation. Geodetic subdivision of $K_6$ (left), result of the first pulling operation (middle), graph after the second pulling operation (right).}
	\label{fig:k6-pull}
\end{figure}

\subsection{Observation on Regular Geodetic Graphs}

\observationregular*

The construction is as follows:
Take a finite projective plane of order $k$ (meaning that it consists of  $k^2 + k + 1$ points).
For each line in the plane, create a $(k+1)$-vertex complete subgraph.
Each of these vertices corresponds to one point on the line.
Furthermore, for each point in the plane, create a separate vertex.
Finally, for each line and each point on that line add an edge between the vertex of the complete subgraph corresponding to that line and point and the separate vertex corresponding to the same point.

Note that the construction in~\cite{bosak1978geodetic} does not rely on projective planes.
Instead it uses a clique edge cover of a complete graph.
In general, this leads to graphs that are not regular.
Here the finite projective plane of order $k$ comes to play:
it yields a covering of the $K_{k^2 + k + 1}$ by $k^2 + k + 1$ cliques of size $k + 1$.

\subsection{Source Code and Data}
\label{sec:source-code}

The source code and the list of biconnected geodetic graphs with up to 25 vertices available on \texttt{osf.io}: \url{https://osf.io/9nxzm/?view_only=276eb3107db24271a2083725af554317}.

\end{document}